\documentclass[11pt,a4paper]{article}

\usepackage{url}
\usepackage{amssymb}
\usepackage{pstricks, pst-coil, pst-node, pst-tree, multido}
\newtheorem{DE}{Definition}[section]

\newcommand {\sm} {\setminus}

\usepackage{graphicx}
\usepackage{latexsym} 
\usepackage{theorem} 

\usepackage{ifpdf}
\ifpdf
\fi

\newcommand{\qed}{\relax\ifmmode\hskip2em\Box\else\unskip\nobreak\hfill$\Box$\fi}

\newtheorem{theorem}[DE]{Theorem}
\newtheorem{lemma}[DE]{Lemma}
\newtheorem{CL}{Claim}

{\theoremstyle{break}\theorembodyfont{\rmfamily}}
{\theoremstyle{break}\theorembodyfont{\rmfamily}}

\newcounter{claim}

\newenvironment{claim}[1][]%
{\refstepcounter{claim}\vspace{1ex}\noindent{(\it\arabic{claim}){#1}{}}\it}{\rule[-2.5ex]{0em}{2.5ex}}

\newenvironment{proofclaim}[1][]%
	{\noindent {}{#1}{}}{ This proves~(\arabic{claim}).\vspace{1ex}}

 \newenvironment{proof}[1][]%
 {\noindent {\setcounter{claim}{0}\bf Proof ---
    }{#1}{}}{\hfill$\Box$\vspace{2ex}}

\begin{document}

\title{A structure theorem for graphs with no cycle with a unique
  chord and its consequences} \author{Nicolas
  Trotignon\thanks{CNRS, Universit\'e Paris 7, Paris Diderot, LIAFA, Case
    7014, 75205 Paris Cedex 13,   France.   E-mail:
    nicolas.trotignon@liafa.jussieu.fr.}~~and Kristina Vu\v skovi\'c\thanks{School of
    Computing, University of Leeds, Leeds LS2 9JT, UK. E-mail:
    vuskovi@comp.leeds.ac.uk.  Partially supported by EPSRC grant
    EP/C518225/1 and Serbian Ministry for Science and Technological Development grant 144015G.}}

\date{January 29, 2009}

\maketitle

\begin{abstract}
  We give a structural description of the class~$\cal C$ of graphs
  that do not contain a cycle with a unique chord as an induced
  subgraph.  Our main theorem states that any connected graph in $\cal
  C$ is either in some simple basic class or has a decomposition.
  Basic classes are chordless cycles, cliques, bipartite graphs with
  one side containing only nodes of degree two and induced subgraphs
  of the famous Heawood or Petersen graph. Decompositions are node
  cutsets consisting of one or two nodes and edge cutsets called
  1-joins. Our decomposition theorem actually gives a complete
  structure theorem for ${\cal C}$, i.e. every graph in ${\cal C}$ can
  be built from basic graphs that can be explicitly constructed, and
  gluing them together by prescribed composition operations; and all
  graphs built this way are in~${\cal C}$.

  This has several consequences: an ${\cal O}(nm)$-time algorithm to decide
  whether a graph is in~$\cal C$, an ${\cal O}(n+m)$-time algorithm that
  finds a maximum clique of any graph in $\cal C$ and an ${\cal O}(nm)$-time
  coloring algorithm for graphs in $\cal C$. We prove that every graph
  in $\cal C$ is either 3-colorable or has a coloring with $\omega$
  colors where $\omega$ is the size of a largest clique. The problem
  of finding a maximum stable set for a graph in $\cal C$ is known to
  be NP-hard.
\end{abstract}



\section{Motivation}
\label{s:motiv}

In this paper all graphs are simple.  We give a structural
characterization of graphs that do not contain a cycle with a unique
chord as an induced subgraph. For the sake of conciseness we call
$\cal C$ this class of graph. Our main result, Theorem~\ref{th:1},
states that every connected graph in $\cal C$ is either in some simple
basic class or has a particular decomposition. Basic classes are
chordless cycles, cliques, bipartite graphs with one side containing
only nodes of degree two and graphs that are isomorphic to an induced
subgraph of the famous Heawood or Petersen graph. Our decompositions
are node cutsets consisting of one or two nodes or an edge cutset
called a 1-join. The definitions and the precise statement are given
in Section~\ref{s:main}. The proof is given in
Section~\ref{s:proof}. Both Petersen and Heawood graphs were
discovered at the end of the XIXth century in the research on the four
color conjecture, see~\cite{petersen:98} and~\cite{heawood:90}. It is
interesting to us to have them both as sporadic basic graphs.  Note
that our theorem works in two directions: a graph is in~$\cal C$ if
and only if it can be constructed by gluing basic graphs along our
decompositions (this is proved in Section~\ref{sec:st}). Such
structure theorems are stronger than the usual decomposition theorems
and there are not so many of them (see
\cite{chudnovsky.seymour:excluding} for a survey).  This is our first
motivation.

\rule{0em}{1ex}

Our structural characterization allows us to prove properties of
classical invariants. We prove in Section~\ref{s:bounding} that every
graph $G$ in~$\cal C$ satisfies either $\chi(G) = \omega(G)$ or
$\chi(G) \leq 3$ (where $\chi(G)$ denotes the chromatic number and
$\omega(G)$ denotes the size of a maximum clique). This is a
strengthening of the classical Vizing bound $\chi(G) \leq \omega(G)
+1$.  So this class of graphs belongs to the family of $\chi$-bounded
graphs, introduced by Gy\'arf\'as \cite{gyarfas} as a natural
extension of perfect graphs: a family of graphs ${\cal G}$ is {\em
  $\chi$-bounded} with {\em $\chi$-binding function} $f$ if, for every
induced subgraph $G'$ of $G \in {\cal G}$, $\chi (G') \leq f(\omega
(G'))$. A natural question to ask is: what choices of forbidden
induced subgraphs guarantee that a family of graphs is $\chi$-bounded?
Much research has been done in this area, for a survey see
\cite{rs}. We note that most of that research has been done on classes
of graphs obtained by forbidding a finite number of graphs. Since
there are graphs with arbitrarily large chromatic number and girth
\cite{erdos}, in order for a family of graphs defined by forbidding a
finite number of graphs (as induced subgraphs) to be $\chi$-bounded,
at least one of these forbidden graphs needs to be acyclic. Vizing's
Theorem \cite{vizing} states that for a simple graph $G$, $\Delta (G)
\leq \chi ' (G) \leq \Delta (G)+1$ (where $\Delta (G)$ denotes the
maximum vertex degree of $G$, and $\chi '(G)$ denotes the chromatic
index of $G$, i.e. the minimum number of colors needed to color the
edges of $G$ so that no two adjacent edges receive the same
color). This implies that the class of line graphs of simple graphs is
a $\chi$-bounded family with $\chi$-binding function $f(x)=x+1$. This
special upper bound for the chromatic number is called the {\em Vizing
  bound}.  We obtain the Vizing bound for the chromatic number by
forbidding a family of graphs none of which is acyclic.  Our result is
algorithmic: we provide an ${\cal O}(nm)$ algorithm that computes an
optimal coloring of every graph in~$\cal C$. Furthermore, it is easy
to see that there exists an ${\cal O}(n+m)$ algorithm that computes a
maximum clique for every graph in~$\cal C$; and it follows from a
construction of Poljak~\cite{poljak:74} that finding a maximum stable
set of a graph in~$\cal C$ is NP-hard (see
Section~\ref{s:cliquestable}).  All this is our second motivation.

\rule{0em}{1ex}

A third motivation is the detection of induced subgraphs. A
\emph{subdivisible graph} (\emph{s-graph} for short) is a triple $B =
(V, D, F)$ such that $(V, D \cup F)$ is a graph and $D \cap F =
\emptyset$.  The edges in $D$ are said to be \emph{real edges of $B$}
while the edges in $F$ are said to be \emph{subdivisible edges of
  $B$}.  A \emph{realisation} of $B$ is a graph obtained from $B$ by
subdividing edges of $F$ into paths of arbitrary length (at least
one).  The problem $\Pi_B$ is the decision problem whose input is a
graph $G$ and whose question is "Does $G$ contain a realisation of $B$
as an induced subgraph?''. In the discussion below, by ``detection
problem'', we mean ``problem $\Pi_B$ for some fixed s-graph $B$''.
This is restrictive since a lot of detection problems of great
interest (such as the detection of odd holes, where a \emph{hole} is
an induced cycle of length at least four) are not of that kind.

Let $H_{1|1}$ be the s-graph on nodes $a, b, c, d$ with real edges
$ab$, $ac$, $ad$ and subdivisible edges $bd$, $cd$. We also define for
$k, l \geq 1$ the s-graph $H_{k|l}$ obtained from $H_{1|1}$ by
subdividing the edge $ab$ into a path of length $k$ and the edge $ac$
into a path of length $l$.  See Fig.~\ref{f:sgraph} where real edges
are represented as straight lines and subdivisible edges as dashed
lines.  The question in Problem $\Pi_{H_{1|1}}$ can be rephrased as
``Does $G$ contain a cycle with a unique chord?''  or ``Is $G$ not
in~$\cal C$?''. The existence of a polynomial time algorithm was an
open question. A consequence of our structural description of~$\cal C$
is an ${\cal O}(nm)$-time algorithm for $\Pi_{H_{1|1}}$ (see
Section~\ref{sec:decomp}).  This is a solution to the recognition
problem for the class~$\cal C$ and it is interesting for reasons
explained below.

\begin{figure}
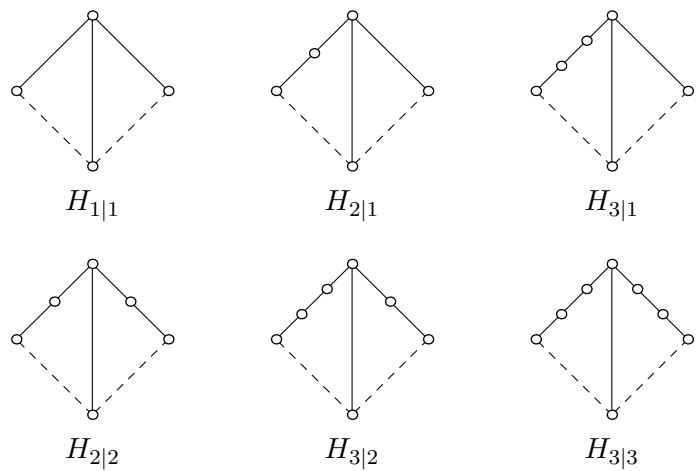

\center
\begin{tabular}{ccccc}
\includegraphics{sgraphs.1}&
\rule{1em}{0ex}&
\includegraphics{sgraphs.2}&
\rule{1em}{0ex}&
\includegraphics{sgraphs.4}\\
$H_{1|1}$&&
$H_{2|1}$&&
$H_{3|1}$
\\
\rule{0em}{1ex}&&&&\\
\includegraphics{sgraphs.3}&
\rule{1em}{0ex}&
\includegraphics{sgraphs.5}&
\rule{1em}{0ex}&
\includegraphics{sgraphs.6}\\
$H_{2|2}$&&
$H_{3|2}$&&
$H_{3|3}$
\end{tabular}
\caption{Some s-graphs\label{f:sgraph}}
\end{figure}

Several problem $\Pi_B$'s can be solved in polynomial time by
non-trivial algorithms (such as detecting pyramids in  
\cite{chudnovsky.c.l.s.v:reco} and thetas in
\cite{chudnovsky.seymour:theta}) and
others that may look similar at first glance are NP-complete (see
\cite{bienstock:evenpair}, \cite{maffray.t:reco}, and
\cite{leveque.lmt:detect} for a survey). A general criterion on an
s-graph that decides whether the related decision problem is
NP-complete or polynomial would be of interest. Our solution of
$\Pi_{H_{1|1}}$ gives some insight in the quest for such a criterion.

A very powerful tool for solving detection problems is the algorithm
three-in-a-tree of Chudnovsky and Seymour (see
\cite{chudnovsky.seymour:theta}). This algorithm decides in time
${\cal O}(n^4)$ whether three given nodes of a given graph $G$ are in an
induced tree of $G$. In~\cite{chudnovsky.seymour:theta}
and~\cite{leveque.lmt:detect} it is observed that every detection
problem $\Pi_B$ for which a polynomial time algorithm is known can be
solved easily by a brute force enumeration or by using
three-in-a-tree. But as far as we can see, three-in-a-tree cannot be
used to solve $\Pi_{H_{1|1}}$, so our solution of $\Pi_{H_{1|1}}$
yields the first example of a detection problem that does not fall
under the scope of three-in-a-tree.  Is there a good reason for that?
We claim that a polynomial time algorithm for $\Pi_{H_{1|1}}$ exists
thanks to what we call \emph{degeneracy}. Let us explain this. Every
statement that we give from here on to the end of the section is under
the assumption that P$\neq$NP.

Degeneracy has to deal with the following question: does putting
bounds on the lengths of the paths in realisations of an s-graph
affect the complexity of the related detection problem?  For upper
bounds, the answer can be found in previous research.  First, putting
upper bounds may turn the complexity from NP-complete to
polynomial. This follows from a simple observation: let $B$ be any
s-graph. A realisation of $B$, where the lengths of the paths arising
from the subdivisions of subdivisible edges are bounded by an integer
$N$, has a number of nodes bounded by a fixed integer $N'$ (that
depends only on $N$ and the size of $B$). So, such a realisation can
be detected in time ${\cal O}(n^{N'})$ by a brute force enumeration.  But
surprisingly, putting upper bounds in another way may also turn the
complexity from polynomial to NP-complete: in
\cite{chudnovsky.c.l.s.v:reco}, a polynomial time algorithm for
$\Pi_{K}$ is given, while in \cite{maffray.t.v:3pcsquare} it is proved
that $\Pi_{K'}$ is NP-complete, where $K, K'$ are the s-graphs
represented in Figure~\ref{f:pp}. Note that $\Pi_{K}$ is usually
called the pyramid (or 3PC($\Delta, \cdot$)) detection problem.

\begin{figure}
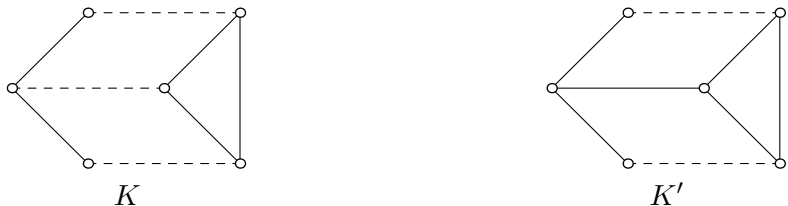

\center
\begin{tabular}{ccc}
\includegraphics{sgraphs.7}&
\rule{8em}{0ex}&
\includegraphics{sgraphs.8}\\
$K$&&
$K'$
\end{tabular}
\caption{Some s-graphs\label{f:pp}}
\end{figure}

Can putting lower bounds turn the complexity from polynomial to
NP-complete? Our recognition algorithm for $\cal C$ shows that the
answer is yes since in Section~\ref{s:bounds} we also prove that the
problem $\Pi_{H_{3|3}}$ is NP-complete. A realisation of $H_{3|3}$ is
simply a realisation of $H_{1|1}$ where every subdivisible edge is
subdivided into a path of length at least three. We believe that a
satisfactory structural description of the class~$\cal C'$ of graphs
that do not contain a realisation of $H_{3|3}$ is hopeless because
$\Pi_{H_{3|3}}$ is NP-complete. So why is there a decomposition
theorem for~$\cal C$~?  Simply because degenerate small graphs like
the \emph{diamond} (that is the cycle on four nodes with exactly one
chord) are forbidden in~$\cal C$, not in~$\cal C'$, and this helps a
lot in our proof of Theorem~\ref{th:1} (the decomposition theorem for
${\cal C}$). This is what we call the \emph{degeneracy} of the
class~$\cal C$. It is clear that degeneracy can help in solving
detection problems, and our results give a first example of this
phenomenon.

So the last question is: can putting lower bounds turn the
complexity from NP-complete to polynomial? We do not know the answer.
Also, we were not able to solve the following questions: what is the
complexity of the problems $\Pi_{H_{2|1}}$, $\Pi_{H_{3|1}}$,
$\Pi_{H_{2|2}}$ and~$\Pi_{H_{3|2}}$? The related classes of graphs
are not degenerate enough to allow us to decompose, and they are too
degenerate to allow us to find an NP-completeness proof.

\rule{0em}{1ex}

A fourth motivation is that our class ${\cal C}$ is related to well
studied classes.  It is a generalization of strongly balanceable
graphs, see \cite{ccv} for a survey.  A bipartite graph is
\emph{balanceable} if there exists a signing of its edges with $+1$
and $-1$ so that the weight of every hole is a multiple of~4.  A
bipartite graph is \emph{strongly balanceable} if it is balanceable
and it does not contain a cycle with a unique chord.  There is an
excluded induced subgraph characterization of balanceable bipartite
graphs due to Truemper \cite{truemper}. A \emph{wheel} in a graph
consists of a hole $H$ and a node $v$ that has at least three
neighbors in $H$, and the wheel is \emph{odd} if $v$ has an odd number
of neighbors in $H$. In a bipartite graph $G$, a \emph{3-odd-path
  configuration} consists of two nonadjacent nodes $u$ and $v$ that
are on opposite sides of the bipartition of $G$, together with three
internally node-disjoint $uv$-paths, such that there are no other
edges in $G$ among the nodes of the three paths. A bipartite graph is
balanceable if and only if it does not contain an odd wheel nor a
3-odd-path configuration \cite{truemper}. So a bipartite graph is
strongly balanceable if and only if it does not contain a 3-odd-path
configuration nor a cycle with a unique chord.

A bipartite graph is \emph{restricted balanceable} if there exists a
signing of its edges with $+1$ and $-1$ so that the weight of every
cycle is a multiple of~4. Conforti and Rao \cite{cr} show that a
strongly balanceable graph is either restricted balanceable or has a
1-join, which enables them to recognize the class of strongly
balanceable graphs (they decompose along 1-joins, and then directly
recognize restricted balanceable graphs).  A bipartite graph is
\emph{2-bipartite} if all the nodes in one side of the bipartition
have degree at most 2. Yannakakis \cite{yan} shows that a restricted
balanceable graph is either 2-bipartite or has a 1-cutset or a 2-join
consisting of two edges (this is an edge cutset that consists of two
edges that have no common endnode), and hence obtains a linear time
recognition algorithm for restricted balanceable graphs.

We note that the basic graphs from our decomposition theorem that do
not have any of our cutsets, and are balanceable, are in fact
2-bipartite.

Class ${\cal C}$ is contained in another well studied class of graphs,
the {\em cap-free} graphs (where a cap is a graph that consists of a
hole and a node that has exactly two neighbors on this hole, and
these two neighbors are adjacent)~\cite{cckv-cap}. In~\cite{cckv-cap}
cap-free graphs are decomposed with 1-amalgams (a generalization of a
1-join) into triangulated graphs and biconnected triangle-free graphs
together with at most one additional node that is adjacent to all
other nodes of the graph. This decomposition theorem is then used to
recognize strongly even-signable and strongly odd-signable graphs in
polynomial time, where a graph is {\em strongly even-signable} if its
edges can be signed with~0 and~1 so that every cycle of length $\geq
4$ with at most one chord has even weight and every triangle has odd
weight, and a graph is {\em strongly odd-signable} if its edges can be
signed with 0 and 1 so that cycles of length 4 with one chord are of
even weight and all other cycles with at most one chord are of odd
weight.

\section{The main theorem}
\label{s:main}

We say that a graph $G$ {\em contains} a graph $H$ if $H$ is
isomorphic to an induced subgraph of $G$.  A graph $G$ is {\em
  $H$-free} if it does not contain $H$.  For $S \subseteq V(G)$,
$G[S]$ denotes the subgraph of $G$ induced by $S$.  A \emph{cycle} $C$
in a graph $G$ is a sequence of nodes $v_1v_2\ldots v_nv_1$, that are
distinct except for the first and the last node, such that for $i=1,
\ldots ,n-1$, $v_iv_{i+1}$ is an edge and $v_nv_1$ is an edge (these
are the edges of $C$). An edge of $G$ with both endnodes in $C$ is
called a \emph{chord} of $C$ if it is not an edge of $C$. One can
similarly define a path and a chord of a path. In this paper we will
only use what is in literature known as chordless paths, so for the
convenience, in this paper (like in \cite{chudnovsky.r.s.t:spgt}) we
define a path as follows: a \emph{path} $P$ in a graph $G$ is a
sequence of distinct nodes $v_1\ldots v_n$ such that for $i=1, \ldots
,n-1$, $v_iv_{i+1}$ is an edge and these are the only edges of $G$
that have both endnodes in $\{ v_1, \ldots ,v_n \}$.  Such a path $P$
is also called a {\em $v_1v_n$-path}.  A {\em hole} is a chordless
cycle of length at least four.  A \emph{triangle} is a cycle of
length~3.  A \emph{square} is a hole of length~4.  A cycle in a graph
is \emph{Hamiltonian} is every node of the graph is in the cycle. Let
us define our basic classes:

The \emph{Petersen graph} is the graph on nodes $\{a_1, \dots a_5,
b_1, \dots, b_5\}$ so that $\{a_1, \dots, a_5\}$ and $\{b_1, \dots,
b_5\}$ both induce a $C_5$ with nodes in their natural order, and such
that the only edges between the $a_i$'s and the $b_i$'s are $a_1b_1$,
$a_2b_4$, $a_3b_2$, $a_4b_5$, $a_5b_3$. See Fig.~\ref{f:petersen}.

\begin{figure}[h]
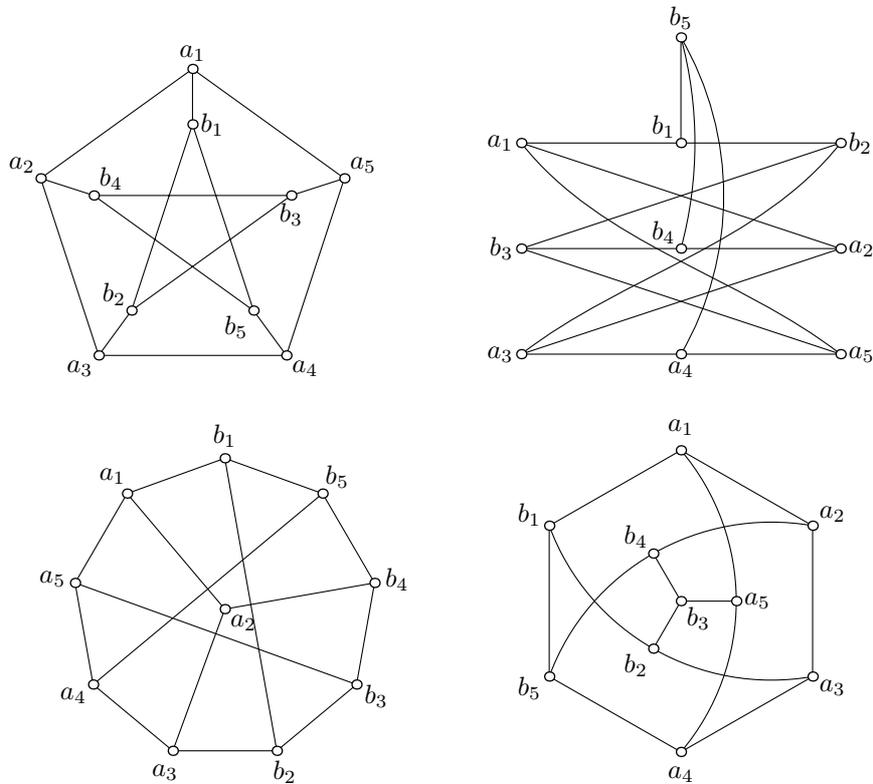

\begin{center}
  \includegraphics{figChord.4} \rule{3em}{0ex}
  \includegraphics{figChord.5}

 \rule{0em}{1ex}

  \includegraphics{figChord.6}  \rule{3em}{0ex}
  \includegraphics{figChord.7}
\end{center}
\caption{Four ways to draw the Petersen graph\label{f:petersen}}
\end{figure}

The \emph{Heawood graph} is the graph on  $\{a_1, \dots ,
a_{14}\}$ so that $\{a_1, \dots , a_{14}\}$ is a Hamiltonian cycle
with nodes in their natural order, and such that the only other edges
are $a_{1}a_{10}$, $a_{2}a_{7}$, $a_{3}a_{12}$, $a_{4}a_{9}$,
$a_{5}a_{14}$, $a_{6}a_{11}$, $a_{8}a_{13}$. See
Fig.~\ref{f:heawood}.

It can be checked that both Petersen and Heawood graph are in~$\cal
C$.  Note that since the Petersen graph and the Heawood graph are both
vertex-transitive, and are not themselves a cycle with a unique chord,
to check that they are in ${\cal C}$, it suffices to delete one node,
and then check that there is no cycle with a unique chord.  Also the
Petersen graph has girth~5 so a cycle with a unique chord in it must
contain at least 8 nodes.  The Heawood graph has girth~6 so a cycle
with a unique chord in it must contain at least 10 nodes.
For the Petersen graph, deleting a node yields an Hamiltonian graph,
and it is easy to check that it does not contain a cycle with a unique
chord.  For the Heawood graph, it is useful to notice that deleting
one node yields the Petersen graph with edges $a_1b_1,b_3b_4,a_3a_4$
subdivided.

\begin{figure}[h]
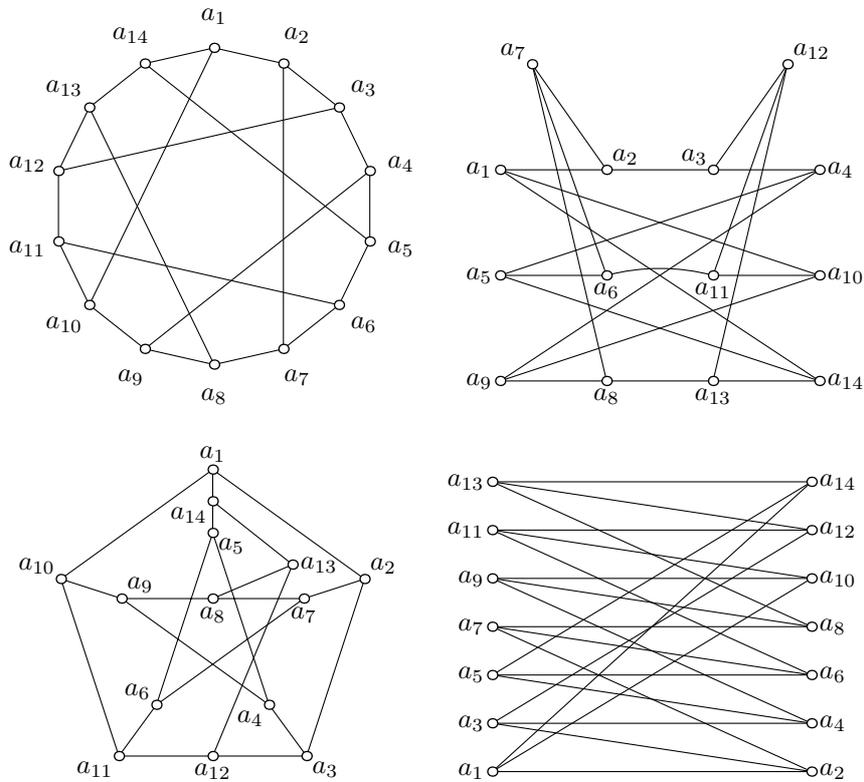

\begin{center}
  \includegraphics{figChord.1}  \rule{1em}{0ex}
  \includegraphics{figChord.2}

 \rule{0em}{1ex}

  \includegraphics{figChord.8} \rule{1em}{0ex}
  \includegraphics{figChord.10}
\end{center}
\caption{Four ways to draw the Heawood graph\label{f:heawood}}
\end{figure}

Let us define our last basic class.  A graph is \emph{strongly
  2-bipartite} if it is square-free and bipartite with bipartition
$(X,Y)$ where $X$ is the set of all degree 2 nodes of $G$ and $Y$ is
the set of all nodes of $G$ with degree at least~3.  A strongly
2-bipartite graph is clearly in $\cal C$ because any chord of a cycle
is an edge linking two nodes of degree at least three, so every cycle
in a strongly 2-bipartite graph is chordless.

We  now define cutsets used in our decomposition theorem:

\begin{itemize}
\item
  A \emph{1-cutset} of a connected graph $G$ is a node $v$ such that
  $V(G)$ can be partitioned into non-empty sets $X$, $Y$ and $\{ v
  \}$, so that there is no edge between $X$ and $Y$.  We say that $(X,
  Y,v)$ is a \emph{split} of this 1-cutset.

\item
  A \emph {proper 2-cutset} of a connected graph $G$ is a pair of
  non-adjacent nodes $a, b$, both of degree at least three, such that
  $V(G)$ can be partitioned into non-empty sets $X$, $Y$ and $\{ a,b \}$ so
  that: $|X|\geq 2$, $|Y| \geq 2$; there are no edges between $X$ and
  $Y$; and both $G[X \cup \{ a,b \}]$ and $G[Y \cup \{ a,b \}]$
  contain an $ab$-path.  We say that $(X, Y, a, b)$ is a \emph{split}
  of this proper 2-cutset.
  
\item
  A \emph{1-join} of a graph $G$ is a partition of $V(G)$ into sets
  $X$ and $Y$ such that there exist sets $A, B$ satisfying:

  \begin{itemize}
  \item
    $\emptyset \neq A \subseteq X$, $\emptyset \neq B \subseteq Y$;
  \item
  $|X| \geq 2$ and $|Y| \geq 2$;
  \item 
    there are all possible edges between $A$ and $B$;
  \item
    there are no other edges between $X$ and $Y$.
  \end{itemize}
  
  We say that $(X, Y, A, B)$ is a \emph{split} of this 1-join.  The
  sets $A, B$ are \emph{special sets} with respect to this 1-join.

  1-Joins were first introduced by Cunningham~\cite{cunningham:1join}.
  In our paper we will use a special type of a 1-join called a {\em
    proper 1-join}: a 1-join such that $A$ and $B$ are stable sets of
  $G$ of size at least two. Note that a square admits a proper 1-join.
\end{itemize}

Our main result is the following decomposition theorem:

\begin{theorem}
  \label{th:1}
  Let $G$ be a connected graph that does not contain a cycle with a
  unique chord. Then either $G$ is strongly 2-bipartite, or $G$ is a
  hole of length at least~7, or $G$ is a clique, or $G$ is an induced
  subgraph of the Petersen or the Heawood graph, or $G$ has a
  1-cutset, a proper 2-cutset, or a proper 1-join.
\end{theorem}

The following intermediate results are proved in the next
section. Theorem~\ref{th:1} follows from Theorems~\ref{th:square}
and~\ref{th:nosquare} (more precisely, it follows
from~\ref{th:nosquare} for square-free graphs, and
from~\ref{th:square} for graphs that contain a square).

\begin{theorem}
  \label{th:triangle}
  Let $G$ be a connected graph that does not contain a cycle with a
  unique chord. If $G$ contains a triangle then either $G$ is a
  clique, or one node of the maximal clique that contains this
  triangle is a 1-cutset of $G$.
\end{theorem}

\begin{theorem}
  \label{th:square}
  Let $G$ be a connected graph that does not contain a cycle with a unique
  chord. Suppose that $G$ contains either a square, the Petersen graph
  or the Heawood graph.  Then either $G$ is the Petersen graph or $G$
  is the Heawood graph or $G$ has a 1-cutset or a proper 1-join.
\end{theorem}

\begin{theorem}
  \label{th:nosquare}
  Let $G$ be a connected square-free graph that does not contain a
  cycle with a unique chord. Then either $G$ is strongly 2-bipartite,
  or $G$ is a hole of length at least~7, or $G$ is a clique or $G$ is
  an induced subgraph of the Petersen or the Heawood graph, or $G$ has
  a 1-cutset or a proper 2-cutset.
\end{theorem}

\section{Proof of Theorems~\ref{th:triangle}, \ref{th:square} and \ref{th:nosquare}}
\label{s:proof} 

We first need two lemmas: 

\begin{lemma}
  \label{c:neighHole}
  Let $G$ be a graph in $\cal C$, $H$ a hole of $G$ and $v$ a
  node of $G\sm H$. Then $v$ has at most two neighbors in $H$, and
  these two neighbors are not adjacent.
\end{lemma}

\begin{proof}
  If $v$ has at least three neighbors in $H$, then $H$ contains a
  subpath $P$ with exactly three neighbors of $v$ and $V(P) \cup
  \{v\}$ induces a cycle of $G$ with a unique chord, a
  contradiction. If $v$ has two neighbors in $H$, they must be
  non-adjacent for otherwise $H\cup \{v\}$ is a cycle with a unique
  chord.
\end{proof}

In a connected graph $G$ two nodes $a$ and $b$ form a {\em 2-cutset}
if $G\setminus \{ a,b \}$ is disconnected.

\begin{lemma}\label{abedge}
  Let $G$ be a connected graph that has no 1-cutset.  If $\{ a,b\}$ is
  a 2-cutset of $G$ and $ab$ is an edge, then $G \not \in {\cal C}$.
\end{lemma}

\begin{proof}
  Suppose $\{ a,b \}$ is a 2-cutset of $G$, and $ab$ is an edge.  Let
  $C_1, \ldots ,C_k$ be the connected components of $G \setminus \{
  a,b\}$.  Since $G$ is connected and has no 1-cutset, for every $i
  \in \{ 1, \ldots ,k\}$, both $a$ and $b$ have a neighbor in
  $C_i$. Let $G'$ be the graph obtained from $G$ by removing the edge
  $ab$.  So for every $i \in \{ 1, \ldots ,k\}$, there is an $ab$-path
  $P_i$ in $G'$ whose interior nodes are contained in $C_i$. Then
  $G[V(P_1) \cup V(P_2)]$ is a cycle with a unique chord, and hence $G
  \not \in {\cal C}$.
\end{proof}

If $H$ is any induced subgraph of $G$ and $D$ is a subset of nodes of
$G\setminus H$, the \emph{attachment of $D$ over $H$} is the set of
all nodes of $H$ that have at least one neighbor in $D$.  When clear
from context we do not distinguish between a graph and its node set,
so we also refer to the attachment of $G[D]$ over $H$.

\subsection*{Proof of Theorem~\ref{th:triangle}}

Suppose $G$ contains a triangle, and let $C$ be a maximal clique of
$G$ that contains this triangle.  In fact, $C$ is unique or otherwise
$G$ contains a diamond.  If $G \neq C$ and if no node of $C$ is a
1-cutset of $G$ then let $D$ be a connected induced subgraph of $G\sm
C$, whose attachment over $C$ contains at least two nodes, and that is
minimal with respect to this property. So, $D$ is a path with one end
adjacent to $a\in C$, the other end adjacent to $b \in C\sm \{a\}$ and
$D \cup \{a, b\}$ induces a chordless cycle.  If $D$ has length zero,
then its unique node (say $u$) must have a non-neighbor $c\in C$ since
$C$ is maximal. Hence, $\{u, a, b, c\}$ induces a diamond, a
contradiction.  If $D$ has length at least one then let $c \neq a, b$
be any node of $C$.  Then the hole induced by $D\cup \{a, b\}$ and
node $c$ contradict Lemma~\ref{c:neighHole}. This proves
Theorem~\ref{th:triangle}.

\subsection*{Proof of Theorem~\ref{th:square}}

    \begin{CL}
      \label{c:noTriangle0}    
      We may assume that $G$ is triangle-free. 
    \end{CL}

    \begin{proof}
      Clear by Theorem~\ref{th:triangle} (note that $G$ cannot be a
      clique).
    \end{proof}

    \begin{CL}
      \label{c:noSquare0}
      We may assume that $G$ is square-free.
    \end{CL}

    \begin{proof}
      Assume $G$ contains a square. Then $G$ contains disjoint sets of
      nodes $A$ and $B$ such that $G[A]$ and $G[B]$ are both stable graphs,
      $|A|, |B| \geq 2$ and every node of $A$ is adjacent to every
      node of $B$. Let us suppose that $A \cup B$ is chosen to be
      maximal with respect to this property. If $V(G) = A\cup B$ then
      $(A, B)$ is a proper 1-join of $G$, so we may assume that there are
      nodes in $G\sm (A\cup B)$.

      \begin{claim}
        \label{c:noSquare1}
        Every component of $G \sm (A \cup B)$ has neighbors only in
        $A$ or only in $B$.
      \end{claim}
    
      \begin{proofclaim}
        Else, let us take a connected induced subgraph $D$ of $G \sm
        (A\cup B)$, whose attachment over $A \cup B$ contains nodes of
        both $A$ and $B$, and that is minimal with respect to this
        property.  So $D = u \dots v$ is a path, no interior node of
        which has a neighbor in $A\cup B$ and there exists $a\in A$,
        $b\in B$ such that $ua, vb \in E(G)$. By
        Claim~\ref{c:noTriangle0}, $u \neq v$, $u$ has no neighbor in
        $B$ and $v$ has no neighbor in $A$. By maximality of $A \cup
        B$, $u$ has a non-neighbor $a' \in A$ and $v$ has a
        non-neighbor $b' \in B$. Now, $D \cup \{a, b, a', b'\}$ is a
        cycle with a unique chord (namely $ab$), a contradiction.
      \end{proofclaim}
  
      From~(\ref{c:noSquare1}), it follows that $G$ has a proper 1-join with
      special sets $A, B$.
    \end{proof}

    Now, we just have to prove the following two claims:

    \begin{CL}
      \label{c:petersen0}
      If $G$ contains the Petersen graph then the theorem holds.
    \end{CL}

    \begin{proof}
      Let $\Pi = \{a_1, \dots a_5, b_1, \dots, b_5\}$ be a set of ten
      nodes of $G$ so that $G[\Pi]$ has adjacencies like in the
      definition of the Petersen graph.  We may assume that there are
      some other nodes in $G$ for otherwise the theorem holds.

      \begin{claim}
        \label{c:pet}
        A node of $G\sm \Pi$ has at most one neighbor in $\Pi$.
      \end{claim}
      
      \begin{proofclaim}
        Otherwise $G$ contains a triangle or a square, contrary to
        Claims~\ref{c:noTriangle0}, \ref{c:noSquare0}.
      \end{proofclaim}
      
      Here below, we use symmetries in the Petersen graph to shorten
      the list of cases. First, the Petersen graph is edge-transitive,
      so up to an automorphism, all edges are equivalent. But also, it
      is ``distance-two-transitive'', meaning that every induced $P_3$
      is equivalent to every other induced $P_3$. To see this, it
      suffices to check that every induced $P_3$ is included in an
      induced $C_5$ and that removing any $P_3$ always yields the same
      graph.

      \begin{claim}
        \label{c:pet2}
        The attachment of any component of $G\sm \Pi$ over $\Pi$
        contains at most one node.
      \end{claim}

      \begin{proofclaim}
        Else, let $D$ be a connected induced subgraph of $G\sm \Pi$
        whose attachment over $\Pi$ contains at least two nodes, and
        that is minimal with respect to this property.  By minimality
        and up to symmetry, $D$ is a path with one end adjacent to
        $a_1$ (and to no other node of $\Pi$ by~(\ref{c:pet})), one
        end adjacent to $x \in \{a_2, a_3\}$ (and to no other node of
        $\Pi$). Moreover, no interior node of $D$ has a neighbor in
        $\Pi$. If $x=a_2$ then $D\cup \{a_1, \dots, a_5\}$ is a cycle
        with a unique chord, a contradiction. If $x=a_3$ then $D\cup
        \{a_1, a_5, b_3, b_4, b_5, a_4, a_3\}$ is a cycle with a
        unique chord, a contradiction again.
      \end{proofclaim}

      From~(\ref{c:pet2}) it follows that $G$ has a 1-cutset.
    \end{proof}

    \begin{CL}
      \label{c:heawood0}
      If $G$  contains the Heawood graph then the theorem holds.
    \end{CL}
    
  \begin{proof}
    Let $\Pi = \{a_1, \dots, a_{14}\}$ be a set of fourteen nodes of
    $G$ so that $G[\Pi]$ has adjacencies like in the definition of the
    Heawood graph. We may assume that there are some other nodes in
    $G$ for otherwise the theorem holds.

    \begin{claim}
      \label{c:hea1}
      A node of $G\sm \Pi$ has at most two neighbors in $\Pi$.
    \end{claim}

    \begin{proofclaim}
      Suppose that some node $v$ in $G\sm \Pi$ has at least two
      neighbors in $\Pi$.  Since the Heawood graph is
      vertex-transitive we may assume $va_1\in E(G)$. By
      Claims~\ref{c:noTriangle0} and \ref{c:noSquare0}, $v$ cannot be
      adjacent to a node at distance 1 or 2 from $v$, namely to any of
      $a_{2}, a_{14}, a_{10}, a_{3}, a_{13}, a_{9}, a_{11}, a_{5},
      a_{7}$. So, the only other possible neighbors are $a_{4}, a_{6},
      a_{8}, a_{12}$. But these four nodes are pairwise at distance
      two in $\Pi$, so by Claim~\ref{c:noSquare0}, $v$ can be adjacent
      to at most one of them.
    \end{proofclaim}

    \begin{claim}
      \label{c:hea2}
      The attachment of any component of $G\sm \Pi$ over $\Pi$
      contains at most one node.
    \end{claim}

    \begin{proofclaim}
      Else, let $D$ be a connected induced subgraph of $G\sm \Pi$
      whose attachment over $\Pi$ contains at least two nodes, and is
      minimal with respect to this property. By minimality and up to
      symmetry, $D$ is a path, possibly of length zero, with one end
      adjacent to $a_1$, one end adjacent to $a_i$ where $i\neq 1$ and
      no interior node of $D$ has neighbors in $\Pi$.  Note that by
      assumption, if $D$ is of length at least one then no end of $D$
      can have more than one neighbor in $\Pi$, because such an end
      would contradict the minimality of $D$.  So, $a_1, a_i$ are the
      only nodes of $\Pi$ that have neighbors in $D$ (when $D$ is of
      length zero this holds by~(\ref{c:hea1}).

      If $i = 2$ then $D\cup \{a_1, a_2, a_7, a_8, a_9, a_{10}\}$ is a
      cycle with a unique chord. If $i \in \{3, 4, 5\}$ then $D\cup
      \{a_i, a_{i+1}, \dots, a_8, a_{13}, a_{14}, a_{1}\}$ is a cycle
      with a unique chord. If $i = 6$ then $D\cup \{a_6, a_{5}, a_{4},
      a_9, a_8, a_{13}, a_{14}, a_1\}$ is a cycle with a unique
      chord. If $i \in \{7, 8\}$ then $D\cup \{a_1, a_{2}, \dots,
      a_i\}$ is a cycle with a unique chord. If $i = 9$ then $D\cup
      \{a_9, \dots, a_{14}\}$ is a cycle with a unique chord. If $i
      \in \{10, 11\}$ then $D\cup \{a_1, a_{2}, a_3, a_4, a_9, a_{10},
      a_i\}$ is a cycle with a unique chord.  If $i \in \{12, 13\}$
      then $D\cup \{a_1, a_{2}, a_3, a_4, a_5, a_6, a_{11}, a_{12},
      a_i\}$ is a cycle with a unique chord.  If $i = 14$ then $D\cup
      \{a_1, a_{2}, a_{7}, a_8, a_{13}, a_{14}\}$ is a cycle with a
      unique chord.  In every case, there is a contradiction.
    \end{proofclaim}
    
    From~(\ref{c:hea2}) it follows that $G$ has a 1-cutset. 
  \end{proof}

  This proves Theorem~\ref{th:square}.

\subsection*{Proof of Theorem~\ref{th:nosquare}}

  \setcounter{CL}{0}

  We consider a graph $G$ containing no cycle with a unique chord and
  no square.  So:

  \begin{CL}
    \label{c:noSquare}
    $G$ is square-free.
  \end{CL}

  Our proof now goes through thirteen claims, most of them of the same
  kind: if some basic graph $H$ is an induced subgraph of $G$, then
  either $G=H$ and so $G$ itself is basic, or some nodes of $G\sm H$
  must be attached to $H$ in a way that entails a proper 2-cutset. At
  the end of this process there are so many induced subgraphs
  forbidden in $G$ that we can prove that $G$ is strongly 2-bipartite.

  \begin{CL}
    \label{c:noTriangle}
    We may assume that $G$ is triangle-free. 
  \end{CL}
      
  \begin{proof}
    Clear by Theorem~\ref{th:triangle}.
  \end{proof}

  \begin{CL}
    \label{c:petersen}
    We may assume that $G$ does not contain the Petersen graph.
  \end{CL}

  \begin{proof}
    Clear by Theorem~\ref{th:square}. Note that $G$ cannot admit a
    proper 1-join since it is square-free. 
  \end{proof}

  \begin{CL}
    \label{c:heawood}
    We may assume that $G$ does not contain the Heawood graph.
  \end{CL}

  \begin{proof}
    Clear by Theorem~\ref{th:square}. 
  \end{proof}

  \begin{CL}
    \label{c:Petersen3pc}
    We may assume that $G$ does not contain the following
    configuration: three node-disjoint paths $X = x \dots x'$, $Y =
    y \dots y'$ and $Z = z \dots z'$, of length at least two and with
    no edges between them. There are four more nodes $a, b, c,
    d$. The only edges except those from the paths are $ax, ay, az,
    bx', by, bz', cx', cy', cz, dx, dy', dz'$.
  \end{CL}

  \begin{proof}
    Let $\Pi = X \cup Y \cup Z \cup \{a, b, c, d\}$. Nodes $a, b,
    c, d, x, x', y, y', z, z'$ are called here the \emph{branch nodes}
    of $\Pi$. It is convenient to notice that $G[\Pi]$ can be obtained
    by subdividing the edges of any induced matching of size three of
    the Petersen graph.  Note also that either $G[\Pi]$ is the Heawood
    graph with one node deleted (when $X, Y, Z$ all have length two), or
    $G[\Pi]$ has a proper 2-cutset (when one of the paths is of length
    at least three, the proper 2-cutset is formed by the ends of that path).
    Hence we may assume that there are nodes in $G\setminus \Pi$.

    \begin{claim}
      \label{c:P3pc1}
      A node of $G \sm \Pi$ has neighbors in at most one of the
      following sets: $X, Y, Z, \{a\}, \{b\}, \{c\}, \{d\}$.
    \end{claim}

    \begin{proofclaim}
      Let $u$ be a node of $G \sm \Pi$.  Note that $a, b, c, d$ are
      pairwise at distance two in $\Pi$, so by Claims~\ref{c:noSquare}
      and \ref{c:noTriangle}, $u$ can be adjacent to at most one of
      them.

      Suppose first that $ux \in E(G)$.  Then $u$ can be adjacent to
      none of $a, d, y, y', z, z'$ by Claims~\ref{c:noSquare}
      and~\ref{c:noTriangle}. If $u$ is adjacent to some other
      branch-node of $G[\Pi]$ distinct from $x'$, then we may assume
      that $u$ is adjacent to one of $b, c$ (say $b$ up to symmetry),
      but then $u$ is not adjacent to $c$ so $u x d y' c z a y b u$ is
      a cycle with a unique chord, a contradiction.  Hence we may
      assume that $u$ has a neighbor $v$ in the interior of $Y$ or $Z$
      (say $Z$ up to symmetry) for otherwise the claim holds. By
      Lemma~\ref{c:neighHole} and since $xdz'Zzcx'Xx$ is a hole, we
      note that $u$ has exactly two neighbors in $X \cup Z$, namely
      $x$ and $v$.  If $vz'\notin E(G)$ then $xuvZzaybz'dx$ is a cycle
      with a unique chord, a contradiction. So $vz'\in E(G)$ and
      $xuvZzcx'bz'dx$ is a cycle with a unique chord (namely $vz'$), a
      contradiction again. Hence we may assume that $ux\notin E(G)$,
      and symmetrically $ux', uy, uy', uz, uz' \notin E(G)$.

      Suppose now that $ua\in E(G)$. Then $u$ cannot be adjacent to
      any other branch-node of $\Pi$ by the discussion above. So we
      may assume that $u$ has a neighbor in the interior of $X, Y, Z$
      (say $Z$ up to symmetry) for otherwise our claim holds. Now we
      define $v$ to be the neighbor of $u$ along $Z$ closest to $z$
      and observe that $uvZzcx'byau$ is a cycle with a unique chord, a
      contradiction. Therefore we may assume that $u$ is not adjacent
      to any branch-node of $\Pi$.

      We may suppose now that $u$ has neighbors in the interior of at
      least two of the paths among $X, Y, Z$ (say $X, Z$ w.l.o.g.) for
      otherwise our claim holds. Let $v$ (resp. $w$) be a neighbor of
      $u$ in $X$ (resp. $Z$). Since $X \cup Z \cup \{d, c\}$ induces a
      hole, by Lemma~\ref{c:neighHole}, $N(u) \cap (X \cup Z) = \{v,
      w\}$.  If $vx \notin E(G)$ then $vuwZz'dxaybx'Xv$ is a cycle
      with unique chord. So $vx \in E(G)$ and symmetrically, $vx', wz,
      wz' \in E(G)$. If $u$ has no neighbor in $Y$ then
      $vuwz'dy'Yybx'v$ is a cycle with a unique chord, so $u$ must
      have at least one neighbor $w'$ in $Y$. By the same discussion
      that we have done above on $X, Z$ we can prove that $w'y, w'y'
      \in E(G)$. Now we observe that $G[\Pi\cup \{u\}]$ is the Heawood
      graph, contradicting Claim~\ref{c:heawood}.
    \end{proofclaim}

    \begin{claim}
      \label{c:P3pc2}
      The attachment of any component of $G\setminus \Pi$ is included
      in one of the sets $X, Y, Z, \{a\}, \{b\}, \{c\}, \{d\}$.
    \end{claim}

    \begin{proofclaim}
      Else let $D$ be a connected induced subgraph of $G\setminus \Pi$
      whose attachment overlaps two of the sets, and is minimal with
      respect this property. By~(\ref{c:P3pc1}), $D$ is a path of
      length at least one, with ends $u, v$ and $u$ (resp. $v$) has
      neighbors in exactly one set $S_u$ (resp. $S_v$) of $X, Y, Z,
      \{a\}, \{b\}, \{c\}, \{d\}$. Moreover, no interior node of $D$
      has a neighbor in $\Pi$. If $S_u = X$ and $S_v = \{a\}$ then let
      $w$ be the neighbor of $u$ closest to $x$ along $X$. We observe
      that $ayYy'dxXwuDva$ is a cycle with a unique chord, a
      contradiction.  Every case where there is an edge between $S_u$
      and $S_v$ is symmetric, so we may assume that there is no edge
      between $S_u$ and $S_v$. If $S_u = X$ and $S_v = Z$ then let $w$
      (resp. $w'$) be the neighbor of $u$ (resp. of $v$) closest to
      $x'$ along $X$ (resp. to $z'$ along~$Z$).  If $w=x$ and $w'=z$
      then $xdy'YyazvDux$ is a cycle with a unique chord, a
      contradiction.  So up to symmetry we may assume $w' \neq z$.
      Hence $wXx'cy'Yybz'Zw'vDuw$ is a cycle with a unique chord, a
      contradiction.  So up to symmetry we may assume that $S_u =\{ a
      \}$ and $S_v = \{ c \}$.  But then $auDvcx'bz'Zza$ is a cycle
      with a unique chord, a contradiction.
    \end{proofclaim}

    By~(\ref{c:P3pc2}), either some component of $G\sm \Pi$ attaches to
    a node of $\Pi$ and there is a 1-cutset, or some component
    attaches to one of $X, Y, Z$ (say $X$ up to symmetry), and $\{x,
    x'\}$ is a proper 2-cutset.
  \end{proof}

  \begin{CL}
    \label{c:hexagonal3PC}
    We may assume that $G$ does not contain the following
    configuration: three node-disjoint paths $X = x \dots x'$, $Y = y
    \dots y'$ and $Z = z \dots z'$, of length at least two, and such
    that the only edges between them are $xy$, $yz$, $zx'$, $x'y'$,
    $y'z'$ and $z'x$.
  \end{CL}

  \begin{proof}
    Note that $G[x, y, z, x', y', z']$ is a hole on six nodes.  Also
    either $G[X\cup Y \cup Z]$ is the Petersen graph with one node
    deleted (when $X, Y, Z$ have length two), or $G[X\cup Y \cup Z]$
    has a proper 2-cutset (when one of the paths is of length at least
    three, the 2-cutset is formed by the ends of that path). Hence we
    may assume that there are nodes in $G\setminus (X\cup Y \cup Z)$.
    
    \begin{claim}
      \label{c:h3pc1}
      A node of $G\sm (X\cup Y \cup Z)$ has neighbors in at most one
      of the sets $X, Y, Z$.
    \end{claim}

    \begin{proofclaim}
      Let $u$ be a node of $G\sm (X \cup Y \cup Z)$.  Note that $u$
      has neighbors in at most one of the following sets: $\{x, x'\}$,
      $\{y, y'\}$, $\{z, z'\}$, for otherwise $G$ contains a triangle
      or a square, contradicting  Claims~\ref{c:noSquare} and
      \ref{c:noTriangle}.
 
      If $u$ has at least two neighbors among $x, x', y, y', z, z'$
      then we may assume by the paragraph above $ux, ux' \in
      E(G)$. Since $X\cup Y$ and $X\cup Z$ both induce holes, every
      node in $X\cup Y \cup Z$ is in a hole going through $x, x'$. So,
      by Lemma~\ref{c:neighHole} $u$ has no other neighbors in $X\cup
      Y \cup Z$. Hence, from here on, we assume that $u$ has at most
      one neighbor among $x, x', y, y', z, z'$.

      If $ux \in E(G)$ then we may assume that $u$ has neighbors in
      one of $Y, Z$, say $Z$ up to symmetry.  Let $v \in Z$ be a
      neighbor of $u$.  Then by Lemma~\ref{c:neighHole}, since $X\cup
      Z$ induces a hole, $v$ and $x$ are the only neighbors of $u$ in
      $X\cup Z$.  So, $xuvZz'y'x'Xx$ is a cycle with a unique chord, a
      contradiction. Hence we may assume that $u$ has no neighbors
      among $x, x', y, y', z, z'$.

      If $u$ has neighbors in the interior of at most one of $X, Y, Z$
      our claim holds, so let us suppose that $u$ has neighbors in the
      interior of $X$ and the interior of $Y$. Since $X\cup Y$ induces
      a hole, by Lemma~\ref{c:neighHole}, $u$ has a unique neighbor
      $v\in X$ and a unique neighbor $w\in Y$.  If $u$ has no neighbor
      in $Z$ then $xXvuwYyzZz'x$ is a cycle with a unique chord, a
      contradiction. So $u$ has a neighbor $w'\in Z$ that is unique by
      Lemma~\ref{c:neighHole}.  If $vx, wy \notin E(G)$ then
      $x'zyxz'y'YwuvXx'$ is a cycle with a unique chord, a
      contradiction. So, up to a symmetry we may assume $vx \in
      E(G)$. If $wy'\notin E(G)$ then $x'y'z'xyYwuvXx'$ is a cycle
      with a unique chord, a contradiction, so $wy'\in E(G)$. By the
      same argument, we can prove that $w'z\in E(G)$. If $vx', wy, w'z' \in
      E(G)$ then we observe that $G[X\cup Y \cup Z \cup \{u\}]$ is the
      Petersen graph, contradicting Claim~\ref{c:petersen}. If $vx',
      wy, w'z' \notin E(G)$ then we observe that the three paths
      $vXx'$, $wYy$, $w'Zz'$ and nodes $u, z, y', x$ have the same
      configuration as those in Claim~\ref{c:Petersen3pc}, a
      contradiction. So, we may assume that $vx'\in E(G)$ and $wy
      \notin E(G)$. But then, $wuvxyzx'y'w$ is a cycle with a unique
      chord, a contradiction.
    \end{proofclaim}

    \begin{claim}
      \label{c:h3pc2}
      The attachment of any component of $G\setminus (X\cup Y \cup Z)$
      is included in one of the sets $X, Y, Z$.
    \end{claim}

    \begin{proofclaim}
      Else let $D$ be a connected induced subgraph of $G\setminus (X
      \cup Y \cup Z)$, whose attachment overlaps two of the sets, and
      is minimal with this property. By~(\ref{c:h3pc1}), $D$ is a path
      of length at least one, with ends $u, v$ and no interior node of
      $D$ has a neighbor in $X\cup Y \cup Z$. We may assume that $u$
      has neighbors only in $X$ and $v$ only in $Y$. Let $u'$
      (resp. $v'$) be the neighbor of $u$ (resp. of $v$) closest to
      $x$ along $X$ (resp. to $y$ along $Y$). If $u'\neq x'$ and
      $v'\neq y'$ then $zyYv'vDuu'Xxz'Zz$ is a cycle with a unique
      chord, a contradiction, so we may assume $v'=y'$.  Let $u''$ be
      the neighbor of $u$ closest to $x'$ along $X$.  If $u'' \neq x$
      then $uu''Xx'zZz'y'vDu$ is a cycle with a unique chord a
      contradiction. So, $u'' = x$ and $uxz'Zzx'y'vDu$ is a cycle with
      a unique chord, a contradiction.
    \end{proofclaim}

    By~(\ref{c:h3pc2}) one of $\{x, x'\}$, $\{y, y'\}$, $\{z, z'\}$ is
    a proper 2-cutset of $G$.
  \end{proof}

  \begin{CL}
    \label{c:octogonal4PC}
    We may assume that $G$ does not contain the following
    configuration: four node-disjoint paths $X = a_1 \dots a_5$, $Y
    = a_2 \dots a_6$, $Z = a_3 \dots a_7$ and $T = a_4 \dots a_8$, of
    length at least two, and such that the only edges between them are
    $a_1a_2$, $a_2a_3$, $a_3a_4$, $a_4a_5$, $a_5a_6$, $a_6a_7$,
    $a_7a_8$ and $a_8a_1$.
  \end{CL}

  \begin{proof}
    Either $G[X\cup Y \cup Z \cup T]$ is obtained from the Heawood
    graph by deleting two adjacent nodes (when $X, Y, Z, T$ have
    length two), or $G[X\cup Y \cup Z \cup T]$ has a proper 2-cutset
    (when one of the paths is of length at least three, the 2-cutset
    is formed by the ends of that path).  Hence we may assume that
    there are nodes in $G\setminus (X\cup Y \cup Z \cup T)$.

    \begin{claim}
      \label{c:o4pc1}
      A node of $G\sm (X\cup Y \cup Z \cup T)$ has at most two
      neighbors in $X\cup Y \cup Z \cup T$.\hspace{-1em}
    \end{claim}

    \begin{proofclaim}
      For suppose that a node $u$ of $G\sm (X\cup Y \cup Z \cup T)$
      has at least three neighbors in $X\cup Y \cup Z \cup T$. Since
      every pair of path from $X, Y, Z, T$ can be embedded in a hole
      (for instance, $X\cup Y$ or $X \cup Z \cup \{a_2, a_6\}$ are
      holes, and the other cases are symmetric), by
      Lemma~\ref{c:neighHole}, the neighbors of $u$ lie on three or
      four paths and every path contains at most one neighbor of $u$.

      Suppose $u$ is adjacent to one of the $a_i$'s, say $a_1$. Then
      by Lemma~\ref{c:neighHole} $u$ has at most one neighbor in $Y
      \cup T$ since $Y \cup T \cup \{a_1, a_5\}$ is a hole.  So up to
      symmetry we assume that $u$ has a neighbor $v$ in $Y$, no
      neighbor in $T$, and so $u$ must have a neighbor $w$ in $Z$.  By
      Lemma~\ref{c:neighHole} applied to the hole $X\cup Z \cup \{a_2,
      a_6\}$ and node $u$, $v$ must be in the interior of $Y$. If
      $w\neq a_3$ then $w Z a_7 a_8 T a_4 a_5 a_6 Y v u w$ is a cycle
      with a unique chord. So $w = a_3$ and $G$ contains a square, a
      contradiction to Claim~\ref{c:noSquare}. Hence we may assume
      that $u$ has no neighbors among the $a_i$'s.

      Up to symmetry we assume that $u$ has neighbors $x\in X$, $y \in
      Y$, $z\in Z$. These neighbors are unique and are in the interior
      of their respective paths. So $uxXa_5a_6Ya_2a_3Zzu$ is a cycle
      with a unique chord (namely $uy$), a contradiction.
    \end{proofclaim}

    \begin{claim} 
      \label{c:o4pc2}
      The attachment of any component of $G\setminus (X\cup Y \cup Z
      \cup T)$ is included in one of the sets $X, Y, Z, T$.
    \end{claim}

    \begin{proofclaim}
      Else let $D$ be a connected induced subgraph of $G\setminus (X
      \cup Y \cup Z \cup T)$, whose attachment overlaps two of the
      sets, and is minimal with respect this property.  By the choice
      of $D$, the following hold.  $D$ is a path, possibly of length
      zero, with ends $u, v$, and we may assume up to symmetry that
      $u$ has neighbors in $X$ and that $v$ has neighbors in $Y$ or in
      $Z$. No interior node of $D$ has neighbors in $X \cup Y \cup Z
      \cup T$.  If $u\neq v$ then $u$ has neighbors only in $X$ and
      $v$ only in $Y$ or in $Z$.  If $u=v$ then by~(\ref{c:o4pc1}) $u$
      has  neighbors only in $X \cup Y$ or only in $X \cup Z$.
   
      If $v$ has neighbors in $Y$ then let $x$ be the neighbor of $u$
      closest to $a_5$ along $X$ and $y$ be the neighbor of $v$
      closest to $a_6$ along $Y$. If $x=a_1$ and $y=a_2$ then $D \cup
      \{a_1, \dots ,a_8\}$ is a cycle with a unique chord, so up to
      symmetry we may assume $x \neq a_1$. But then, $u x X a_5 a_4 T
      a_8 a_7 a_6 Y y v D u$ is a cycle with a unique chord, a
      contradiction.
 
      So $v$ has neighbors in $Z$. We claim that $v$ has a unique
      neighbor $z$ in $Z$, that is in the interior of $Z$, and that
      $Z$ has length two. Else, up to the symmetry between $a_3$ and
      $a_7$ we may assume that the neighbor $z$ of $v$ closest to
      $a_3$ along $Z$ is not $a_7$ and is not adjacent to $a_7$.  Let
      $x$ be the neighbor of $u$ closest to $a_5$ along $X$. If
      $x=a_1$ then $a_1Xa_5a_6Ya_2a_3ZzvDua_1$ is a cycle with a
      unique chord.  So $x \neq a_1$.  If $v$ is not adjacent to
      $a_7$, then $a_5a_6a_7a_8Ta_4a_3ZzvDuxXa_5$ is a cycle with
      unique chord, a contradiction. So $v$ is adjacent to $a_7$. In
      particular, $u \neq v$.  By~(\ref{c:o4pc1}), $a_7$ and $z$ are
      the only neighbors of $v$ in $Z$. But then $a_5a_6a_7ZzvDuxXa_5$
      is a cycle with unique chord, a contradiction.  So, our claim is
      proved. Similarly, it can be proved that $u$ has a unique
      neighbor $x$ in $X$, that this neighbor is in the interior of
      $X$, and that $X$ has length two.  We observe that the three
      paths $xuDvz$, $Y$, $T$ and nodes $a_1, a_3, a_7, a_5$ have the
      same configuration as those in Claim~\ref{c:Petersen3pc}, a
      contradiction.
    \end{proofclaim}

    By~(\ref{c:o4pc2}), one of 
    $\{a_1, a_5\}$, $\{a_2, a_6\}$, $\{a_3, a_7\}$, $\{a_4, a_8\}$ is
    a proper 2-cutset.
  \end{proof}

  \begin{CL}
    \label{c:5PC}
    We may assume that $G$ does not contain the following
    configuration: five paths $P_{13} = a_1 \dots a_3$, $P_{15} = a_1
    \dots a_5$, $P_{48} = a_4 \dots a_8$, $P_{37} = a_3 \dots a_7$,
    $P_{57} = a_5 \dots a_7$, node disjoint except for their ends, of
    length at least two, and such that $G[P_{13} \cup P_{15} \cup
      P_{37} \cup P_{57}]$ is a hole and the only edges between this
    hole and $P_{48}$ are $a_3a_4$, $a_4a_5$, $a_7a_8$ and $a_8a_1$.
  \end{CL}

  \begin{proof}
    We put $\Pi = P_{13} \cup P_{15} \cup P_{48} \cup P_{37} \cup
    P_{57}$.  Either $G[\Pi]$ is the Heawood graph with three nodes
    inducing a $P_3$ deleted (when the five paths have length two), or
    $G[\Pi]$ has a proper 2-cutset (when one of the paths is of length
    at least three, the 2-cutset is formed by the ends of that path).
    Hence we may assume that there are nodes in $G\setminus \Pi$.  

    \begin{claim}
      \label{c:5pc1}
      A node of $G\sm \Pi$ has at most two neighbors in $\Pi$. 
    \end{claim}

    \begin{proofclaim}
      Let $u$ be a node of $G \sm \Pi$ and suppose that $u$ has more
      than two neighbors in~$\Pi$.  By Lemma~\ref{c:neighHole} and
      since $P_{13} \cup P_{15} \cup P_{37} \cup P_{57}$ is a hole,
      $u$ has at most two neighbors among these paths. So, $u$ must
      have one neighbor in $P_{48}$, and this neighbor is unique since
      the union of $P_{48}$ with any of the other paths yields a
      hole. For the same reason, $u$ has a unique neighbor in exactly
      two paths among $P_{13}, P_{15}, P_{37}, P_{57}$. So there are
      two cases up to symmetry: either $u$ has neighbors in two paths
      among $P_{13}, P_{15}, P_{37}, P_{57}$ that have a common end,
      or $u$ has neighbors in two paths among $P_{13}, P_{15}, P_{37},
      P_{57}$ that have no common ends.

      In the first case, we may assume that $u$ has neighbors $x \in
      P_{37}$, $y \in P_{48}$ and $z \in P_{13}$. Note that $x \neq
      a_3$ and $z \neq a_3$, for otherwise $P_{13}$ or $P_{37}$ would
      contain two neighbors of $u$.  Suppose $y \neq a_4$. If $z \neq
      a_1$ then $x u y P_{48} a_8 a_1 P_{15} a_5 P_{57} a_7 P_{37} x$
      is a cycle with a unique chord, a contradiction.  If $z=a_1$
      then by Claim~\ref{c:noSquare}, $x \neq a_7$ and hence
      $xuyP_{48} a_8 a_1 P_{13} a_3 P_{37} x$ is a cycle with a unique
      chord, a contradiction.  So $y = a_4$. But then, since $G$ does
      not contain a square, $x a_3 \notin E(G)$ and hence $u x P_{37}
      a_7 a_8 a_1 P_{13} a_3 a_4 u$ is a cycle with a unique chord, a
      contradiction.

      In the second case, we may assume that $u$ has neighbors $x$ in
      $P_{13}$, $y$ in $P_{48}$ and $z$ in $P_{57}$.  If $x=a_1$ then
      the previous case applies.  Hence we may assume $x\neq a_1$ and
      symmetrically $z \neq a_5$.  So, $u x P_{13} a_3 a_4 P_{48} a_8
      a_7 P_{57} z u$ is a cycle with a unique chord (namely $uy$).
    \end{proofclaim}

    \begin{claim}
      \label{c:5pc2}
      The attachment of any component of $G\setminus \Pi$ is included
      in one of the sets $P_{13}, P_{15}, P_{37}, P_{57}, P_{48}$.
    \end{claim}

    \begin{proofclaim}
      Else let $D$ be a connected induced subgraph of $G\setminus
      \Pi$, whose attachment is not contained in one of the sets, and
      is minimal with respect to this property. By the choice of $D$
      the following hold.  $D$ is a path, possibly of length zero,
      with ends $u, v$, where $u$ has neighbors in one of the sets
      $P_{13}, P_{15}, P_{37}, P_{57}, P_{48}$ that we denote by
      $X_u$, and $v$ has neighbors in another one, say $X_v$. No
      interior node of $D$ has neighbors in $\Pi$.  If $u\neq v$ then
      $u$ has neighbors only in $X_u$ and $v$ only in $X_v$.  If
      $u=v$ then by~(\ref{c:5pc1}) $u$ has neighbors only in $X_u
      \cup X_v$.

      If $X_u = P_{48}$ then up to symmetry we may assume $X_v =
      P_{57}$. Let $x$ be the neighbor of $u$ closest to $a_8$ along
      $P_{48}$. If $x\neq a_4$ then let $y$ be the neighbor of $v$
      closest to $a_7$ along $P_{57}$. Then $u D v y P_{57}
      a_7 P_{37} a_3 P_{13} a_1 a_8 P_{48} x u$ is a cycle with a
      unique chord.  If $x = a_4$ then let $y$ be the neighbor of $v$
      closest to $a_5$ along $P_{57}$. Then $u D v y P_{57}
      a_5 P_{15} a_1 P_{13} a_3 a_4 u$ is a cycle with a unique
      chord. So, $X_u \neq P_{48}$, and symmetrically $X_v \neq
      P_{48}$.

      If $X_u$, $X_v$ are paths with a common end then we may assume
      $X_u = P_{37}$ and $X_v = P_{57}$. Let $x$ be the neighbor of
      $u$ closest to $a_3$ along $P_{37}$ and $y$ the neighbor of $v$
      closest to $a_5$ along $P_{57}$. We note that $x, y \neq a_7$
      for otherwise the attachment of $D$ is a single path $P_{37}$ or
      $P_{57}$ contrary to the definition of~$D$. So, $x u D v y
      P_{57} a_5 a_4 P_{48} a_8 a_1 P_{13} a_3 P_{37} x$ is a cycle
      with a unique chord.
    
      If $X_u$, $X_v$ are paths with no common end then we may assume
      $X_u = P_{13}$ and $X_v = P_{57}$. We claim that $u$ has a
      unique neighbor in $P_{13}$, that is in the interior of
      $P_{13}$, and that $P_{13}$ has length two. Else, up to the
      symmetry between $a_1$ and $a_3$ we may assume that the neighbor
      $x$ of $u$ closest to $a_3$ along $P_{13}$ is not $a_1$ and is
      not adjacent to $a_1$.  Let $y$ be the neighbor of $v$ closest
      to $a_5$ along $P_{57}$.  If $y=a_7$ then the previous case
      applies, i.e. the neighbors of $u$ and $v$ in $\Pi$ are
      contained in the paths with a common end. So $y\neq a_7$.  If
      $u$ is not adjacent to $a_1$, then
      $xuDvyP_{57}a_5P_{15}a_1a_8P_{48}a_4a_3P_{13}x$ is a cycle with
      a unique chord. So $u$ is adjacent to $a_1$.  By~(\ref{c:5pc1}),
      $a_1$ and $x$ are the only neighbors of $u$ in $\Pi$.  But then
      $xuDvyP_{57}a_5P_{15}a_1P_{13}x$ is a cycle with a unique chord.
      Our claim is proved, and similarly we can prove that $v$ has a
      unique neighbor in $P_{57}$, that this neighbor is in the
      interior of $P_{57}$ and that $P_{57}$ has length two. Now we
      observe that the paths $x u D v y$, $P_{15}, P_{37}, P_{48}$
      have the same configuration as those in
      Claim~\ref{c:octogonal4PC}, a contradiction.
    \end{proofclaim}

    By~(\ref{c:5pc2}), one of $\{a_1, a_5\}$, $\{a_1, a_3\}$, $\{a_4,
    a_8\}$, $\{a_5, a_7\}$ $\{a_3, a_7\}$ is a proper 2-cutset.
  \end{proof}

  \begin{CL}
    \label{c:2chords}
    We may assume that $G$ does not contain a cycle with exactly two
    chords.
  \end{CL}

  \begin{proof}
    For let $C$ be a cycle in $G$ with exactly two chords $ab, cd$. We
    may assume up to the symmetry between $c$ and $d$ that $a, c, b,
    d$ appear in this order along $C$ for otherwise there is a cycle
    with a unique chord. We denote by $P_{ac}$ the unique path in $C$
    from $a$ to $c$ that does not go through $b, d$. We define
    similarly $P_{cb}$, $P_{bd}$, $P_{da}$.  We assume that $C$ is a
    cycle with exactly two chords in $G$ that has the fewest number of
    nodes.

    If $P_{cb}$ has length one then $P_{ac} \cup P_{bd}$ is a cycle
    with a unique chord unless $P_{ad}$ has also length one. But then
    $G[a, b, c, d]$ is a square or contains a triangle, a
    contradiction to Claims~\ref{c:noSquare} and \ref{c:noTriangle}. So
    $P_{cb}$ has length at least two and symmetrically, $P_{ac}$,
    $P_{bd}$, $P_{da}$ have all length at least two. 

    Note  that either $C$ is the Petersen graph with two adjacent
    nodes deleted (when $C$ is on eight nodes), or $C$ has a proper
    2-cutset (when $C$ is on at least nine nodes). Hence we may assume
    that there are nodes in $G\setminus C$.

    \begin{claim}
      \label{c:2chords1}
      A node of $G\sm C$ has at most two neighbors in $C$, and if it
      has two neighbors in $C$ then these two neighbors are not
      included in one of the sets $P_{ac}, P_{cb}, P_{bd}, P_{da}$.
    \end{claim}
    
    \begin{proofclaim}
      Let $u$ be a node of $G\sm C$ that has at least three neighbors
      in $C$. Note that by Lemma~\ref{c:neighHole}, $u$ has at most
      one neighbor in each of $P_{ac}$, $P_{cb}$, $P_{bd}$, $P_{da}$
      because the union of any two of them forms a hole.  So, up to
      symmetry we may assume that $u$ has neighbors $x \in P_{ad}$, $y
      \in P_{ac}$, $z \in P_{bd}$ (and possibly one more in
      $P_{cb}$). If $y = c$ then $x\neq d$ and $xd \notin E(G)$ for
      otherwise $G$ contains a square or a triangle, contradicting
      Claims~\ref{c:noSquare} and \ref{c:noTriangle}, and hence $u c d
      P_{bd} b a P_{ad} x u$ has a unique chord (namely $uz$). So
      $c\neq y$ and symmetrically, $b\neq z$. Hence $y u z P_{bd} d
      P_{ad} a P_{ac} y $ is a cycle with a unique chord (namely
      $ux$), a contradiction. So $u$ has at most two neighbors in $C$.

      Let $x$ and $y$ be two neighbors of $u$ in $C$, and suppose that
      they both belong to the same path, say $P_{bd}$. W.l.o.g.\ $x$
      is closer to $b$ on $P_{bd}$. By Claims~\ref{c:noSquare} and
      \ref{c:noTriangle}, the $xy$-subpath $P$ of $P_{bd}$ is of
      length greater than 2. Let $C'$ be the cycle induced by $(C
      \setminus P) \cup \{ x,y \}$. Then $C'$ is a cycle with exactly
      two chords that has fewer nodes than $C$, contradicting our
      choice of $C$.
    \end{proofclaim}

    \begin{claim}
      \label{c:2chords2}
      The attachment of any component of $G\setminus C$ is included in
      one of the sets $P_{ac}, P_{cb}, P_{bd}, P_{da}$.
    \end{claim}

    \begin{proofclaim}
      Else let $D$ be a connected induced subgraph of $G\setminus C$,
      whose attachment is not contained in one of the sets and is
      minimal with respect to this property.  By the choice of $D$ the
      following hold.  $D$ is a path possibly of length zero, with
      ends $u, v$, where $u$ has neighbors in one of the sets $P_{ac},
      P_{cb}, P_{bd}, P_{da}$ that we denote by $X_u$, and $v$ has
      neighbors in another one, say $X_v$.  No interior node of $D$
      has neighbors in $C$.  If $u\neq v$ then $u$ has neighbors only
      in $X_u$ and $v$ only in $X_v$.  If $u=v$, then
      by~(\ref{c:2chords1}) $u$ has neighbors only in $X_u \cup X_v$.
      Let $x$ be a neighbor of $u$ in $X_u$, and $y$ a neighbor of $v$
      in $X_v$.  By~(\ref{c:2chords1}), $(N(u) \cup N(v)) \cap C=\{
      x,y \}$.

      If $X_u$ and $X_v$ share a common end then up to symmetry we
      assume $X_u = P_{ac}$, $X_v = P_{ad}$.  Neither $x$ nor $y$
      coincides with $a$ for otherwise the attachment of $D$ over $C$
      is in $P_{ac}$ or $P_{ad}$, contrary to the definition
      of~$D$. So, $u D v y P_{ad} d P_{bd} b P_{bc} c P_{ac} x u$ is a
      cycle with a unique chord.

      So $X_u$ and $X_v$ do not share a common end, hence up to
      symmetry we assume $X_u = P_{ac}$, $X_v = P_{bd}$.  By the
      previous paragraph, we may assume $x \not \in \{ a,c \}$ and $y
      \not \in \{ b,d \}$.  If $xa, yb \notin E(G)$ then $x P_{ac} c
      P_{cb} b a P_{ad} d P_{bd} y v D u x$ is a cycle with a unique
      chord. So, up to symmetry we assume $xa \in E(G)$.  If $y d
      \notin E(G)$ then $x u D v y P_{bd} b a P_{ad} d c P_{ac} x$ is
      a cycle with a unique chord, a contradiction. So, $yd \in E(G)$.
      Since $xabP_{bc}cdyvDux$ cannot be a cycle with a unique chord,
      $xc, yb$ are either both in $E(G)$ or both not in $E(G)$. In the
      first case, the three paths $x u D v y$, $P_{ad}, P_{bc}$ have
      the same configuration as those in Claim~\ref{c:hexagonal3PC}.
      In the second case, the five paths $x u D v y$, $P_{ad}, P_{bc},
      x P_{ac} c, y P_{bd} b$ have the same configuration as those in
      Claim~\ref{c:5PC}.
    \end{proofclaim}
    
    By~(\ref{c:2chords2}) one of 
    $\{a, c\}$, $\{c, b\}$, $\{b, d\}$, $\{d, a\}$ is a proper
    2-cutset.
  \end{proof}

  \begin{CL}
    \label{c:3chords}
    We may assume that $G$ does not contain a cycle with exactly three
    chords.
  \end{CL}

  \begin{proof}
    Let $C$ be a cycle in $G$ with exactly three chords $ab$, $cd$,
    $ef$ say. Up to symmetry we may assume that $a, c, e, b, d, f$
    appear in this order along the cycle and are pairwise distinct for
    otherwise $C$ contains a cycle with a unique chord. We denote by
    $P_{ac}$ the unique path from $a$ to $c$ in $C$ that does not go
    through $e, b, d, f$. We define similarly $P_{ce}, P_{eb}, P_{bd},
    P_{df}, P_{fa}$. If $G[\{a, b, c, d, e, f\}]$ contains only three
    edges then $P_{af} \cup P_{fd} \cup P_{ce} \cup P_{eb}$ is a cycle
    with a unique chord (namely $fe$), a contradiction. Hence, up to
    symmetry we may assume $ac \in E(G)$. Now $P_{af} \cup P_{db} \cup
    P_{ce}$ is a cycle with one, two or three chords : $ac$ and
    possibly $fd$ and $eb$. By Claim~\ref{c:2chords}, this cycle must
    have three chords, so $eb, fd \in E(G)$. Note that $bd \notin
    E(G)$ since $G$ contains no square by Claim~\ref{c:noSquare}, and
    similarly $af, ce \notin E(G)$. Now we observe that the paths
    $P_{af}, P_{ce}, P_{bd}$ have the same configuration as those in
    Claim~\ref{c:hexagonal3PC}, a contradiction.
  \end{proof}

  \begin{CL}
    \label{c:Nchords}
    We may assume that $G$ does not contain a cycle with at least one
    chord.
  \end{CL}

  \begin{proof}
    Let $C$ be a cycle in $G$ with at least one chord $ab$. We choose
    $C$ minimal with this property. Cycle $C$ must have another chord
    $cd$, and we may assume that $a, d, b, c$ are pairwise distinct
    and in this order along $C$ for otherwise $C$ contains a cycle
    with at least one chord that contradicts the minimality of $C$. By
    Claim~\ref{c:2chords}, $C$ must have another chord $ef$, and again
    we may assume that $a, e, d, b, f, c$ are pairwise distinct and in
    this order along $C$ because of the minimality of $C$. By
    Claim~\ref{c:3chords}, $C$ must have again another chord $gh$, and
    again we may assume that $a, g, e, d, b, h, f, c$ are pairwise
    distinct and in this order along $C$ because of the minimality of
    $C$. Now, the path from $a$ to $f$ along $C$ that goes through $c$
    and the path from $e$ to $b$ along  $C$ that goes through $d$
    form a cycle smaller than $C$ with at least one chord (namely
    $cd$), a contradiction. 
  \end{proof}

  A \emph{non-induced path} $P$ in a graph $G$ is a sequence of
  distinct nodes $v_1\ldots v_n$ such that for $i=1, \ldots ,n-1$,
  $v_iv_{i+1}$ is an edge (these are the \emph{edges of the
    path}). There might be other edges: the \emph{chords of the
    path}.

  \begin{CL}
    \label{c:basket}
    We may assume that $G$ does not contain the following
    configuration: five non-induced paths $P = a \dots c$, $Q = a
    \dots c$, $R = b \dots d$, $S = b \dots d$, $X = c \dots d$,
    node-disjoint except for their ends, of length at least one,
    except for $X$ that can be of length zero or more, together with
    edge $ab$.
  \end{CL}

  \begin{proof}
    Let $\Pi = P \cup Q \cup R \cup S \cup X$.  We suppose that $\Pi$
    is chosen subject to the minimality of $X$.  Note that the only
    edges of $G[\Pi]$ are the edges of the non-induced paths $P, Q, R,
    S, X$ and $ab$.  Indeed, every pair of nodes $\{x,y\} \subseteq
    \Pi$ can be embedded into a cycle containing only edges of the
    non-induced paths and $ab$, so if $xy$ is an edge of $G[\Pi]$ that
    contradicts our statement then there is a cycle $C$ of $G[\Pi]$
    such that $xy$ is a chord of $C$, and this contradicts
    Claim~\ref{c:Nchords}.  In particular, the five non-induced paths
    have no chords, so they are in fact paths.  Note that $P, Q, R, S$
    are all of length at least two for if $P$ (say) is of length one,
    then the unique edge of $P$ is a chord of a cycle of $G[\Pi]$, and
    this contradicts Claim~\ref{c:Nchords}.

    We now show that $\{a, c\}$ is a proper 2-cutset of $G$. Assume
    not.  Then there is a path $D=u \ldots v$ in $G \setminus \Pi$
    such that $u$ has a neighbor $x$ in $(P \cup Q) \setminus \{
    a,c\}$ (say in $Q\sm \{a, c\}$) and $v$ has a neighbor $y$ in $(X
    \cup R \cup S) \setminus \{ c\}$.

    Suppose that $y \in X\sm \{c\}$.  Then the five non-induced paths
    $a Q x u D v y$, $a P c X y$, $R$, $S$ and $y X d$ form a
    configuration that contradicts the minimality of~$X$.

    So $y \not \in X \setminus \{ c\}$, and hence w.l.o.g.\ $y \in R
    \setminus \{ d\}$. So $u x Q a P c X d S b R y v D u$ is a cycle
    with at least one chord (namely $ab$), contradicting
    Claim~\ref{c:Nchords}.
  \end{proof}

  \begin{CL}
    \label{c:noI}
    We may assume that $G$ does not contain the following induced
    subgraph (that we call $I$): six nodes $a, b, c, d, e, f$ with the
    following edges: $ab$, $ac$, $ad$, $be$, $bf$.
  \end{CL}

  \begin{proof}
    We may assume that $G$ has no 1-cutset.  Hence, by
    Lemma~\ref{abedge}, $\{a, b\}$ is not a cutset of $G$.  So there
    exists a path in $G \sm \{a, b\}$ with an end having neighbors in
    $\{c, d\}$ and an end having neighbors in $\{e, f\}$.  We choose a
    minimal such path $D = u \dots v$. Up to symmetries and since $G$
    contains no square and no triangle by Claims~\ref{c:noSquare} and
    \ref{c:noTriangle}, we may assume that $u$ has a unique neighbor
    in $\{c, d\}$, and that $v$ has a unique neighbor in $\{e, f\}$.
    W.l.o.g.\ $du, vf \in E(G)$.  Note that from the minimality of
    $D$, $c, e$ have no neighbor in $D$.

    Since $a$ is not a 1-cutset there is a path $F$ in $G\sm a$ with
    one end $y$ adjacent to $c$  and an end $x$
    adjacent to some node $w$ in $D \cup \{e, b, f, d\}$. We choose
    such a path $F$ minimal with respect to this property. If $x$ is
    adjacent to $e$ or $b$ then $G$ contains a cycle with at least one
    chord (namely $ab$), contradicting Claim~\ref{c:Nchords}.  So, $w
    \in D' = duDvf$.  Note that from the minimality of $F$, $e$ has no
    neighbor in $F$.

    Since $b$ is not a 1-cutset, there is in $G\sm b$ a path $H$ with
    an end $z$ adjacent to $e$ and an end $t$ adjacent to some node
    $s$ in $D \cup F \cup \{f, d, a, c\}$.  Let $Q = acyFx$.  If $s\in
    Q$ then $ezHtsQadD'fbe$ is a cycle with at least one chord (namely
    $ab$), contradicting Claim~\ref{c:Nchords}. So $s\notin Q$.  Let
    $Q' = adD'w \sm w$ (note that if $w=d$ then $Q' = a$).  If $s\in
    Q'$ then $sQ'acyFxwD'fbezHts$ is a cycle with at least one chord
    (namely $ab$), contradicting Claim~\ref{c:Nchords}. So $s\notin
    Q'$.  Hence, $s\in fD'w$. 

    Now we observe that the five non-induced paths $bezHts$, $bfD's$,
    $acyFxw$, $adD'w$ and $wD's$ together with edge $ab$ have the same
    configuration as those in Claim~\ref{c:basket}, a contradiction.
  \end{proof}

  We can now prove that $G$ is strongly 2-bipartite.  Indeed, we may
  assume that $G$ has no 1-cutset and $G$ contains no square by
  Claim~\ref{c:noSquare}.  We may assume that $G$ is not a chordless
  cycle because $C_3$ is a clique, $C_4$ is a square, $C_5$, $C_6$ are
  induced subgraphs of the Petersen graph and $C_k$ where $k \geq 7$
  is an output of our theorem.  We may also assume that $G$ is not a
  clique.  Let us call a \emph{branch} of a graph $G$ any path of
  length at least one, whose ends are of degree at least 3, and whose
  interior nodes are of degree~2.  Since $G$ is not a chordless cycle,
  has no 1-cutset and is not a clique of size one or two, it is
  edge-wise partitioned into its branches.  In particular, every node
  of $G$ is of degree at least two.  No branch of $G$ is of length
  one, because such a branch is an edge of $G$ that has both ends of
  degree at least three, and then $G$ contains either a triangle, a
  square or an $I$, and this contradicts Claim~\ref{c:noTriangle},
  \ref{c:noSquare} or \ref{c:noI}.  We may also assume that $G$ has no
  branch of length at least~3 because the ends of such a branch, say
  $a$ and $b$, form a proper 2-cutset (note that $a$ and $b$ cannot be
  adjacent since there is no branch of length 1, and since there is no
  1-cutset, for every connected component $C$ of $G \setminus \{ a,b
  \}$ there is an $ab$-path in $G[ C \cup \{ a,b \} ]$).

  So we proved that every branch of $G$ is of length exactly~2.  This
  implies that the set $X$ of all nodes of $G$ of degree~2 and the set
  $Y$ of all nodes of $G$ with degree at least~3 are stable sets.  So
  $G$ is strongly 2-bipartite.  This proves Theorem~\ref{th:nosquare}.

\section{Structure theorem}
\label{sec:st}

The \emph{block} $G_X$ (resp. $G_Y$) of a graph $G$ with respect to a
1-cutset with split $(X, Y, v)$ is 
$G[X\cup \{ v \} ]$ (resp. $G[Y\cup \{ v\} ]$).

The \emph{block} $G_X$ (resp. $G_Y$) of a graph $G$ with respect to a
1-join with split $(X, Y, A, B)$ is the graph obtained by taking
$G[X]$ (resp. $G[Y]$) and adding a node $y$ complete to $A$ (resp. $x$
complete to $B$). Nodes $x, y$ are called the \emph{markers} of their
respective blocks.

The \emph{block} $G_X$ (resp. $G_Y$) of a graph $G$ with respect to a
proper 2-cutset with split $(X, Y, a, b)$ is the graph obtained 
by taking $G[X\cup \{ a,b\}]$ 
(resp. $G[Y\cup \{ a,b\}]$) and
  adding a node $c$ adjacent to $a, b$. 
Node $c$ is a called the \emph{marker} of the block $G_X$
(resp. $G_Y$).  

A graph is \emph{basic} if it is connected and it is either a clique,
a hole of length at least~7, a strongly 2-bipartite graph, or an
induced subgraph of the Petersen graph or the Heawood graph.  Note
that the square is not basic (but has a proper 1-join).  Every $C_k$,
$k\geq 3$, $k\neq 4$ is basic.

It is sometime useful to prove that every graph in a class has an
\emph{extremal} decomposition, that is a decomposition such that one
of the blocks is basic. With our basic classes, decompositions and
blocks, this is false for graphs in~$\cal C$.  The graph in
Fig.~\ref{f:noextr} is a counter-example. This graph has no proper
2-cutset and a unique  1-join.  No block with respect to this
proper 1-join is basic, but both blocks have a proper 2-cutset.

\begin{figure}
\center
\includegraphics{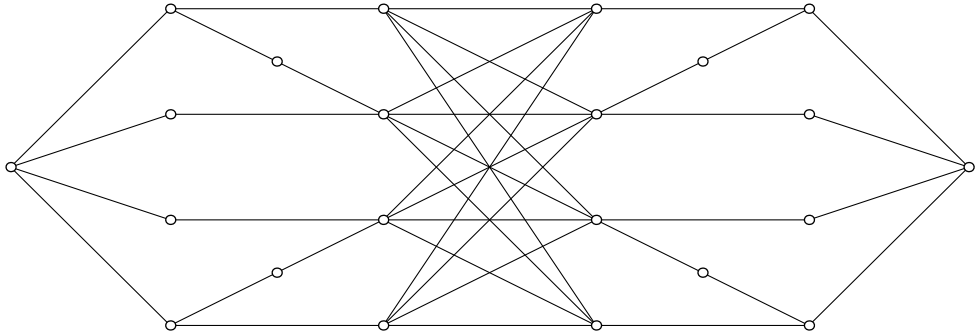}
\caption{A graph in $\cal C$ with no extremal decomposition\label{f:noextr}}
\end{figure}

\begin{lemma}\label{iff0}
Let $G_X$ and $G_Y$ be the blocks of decomposition of $G$
w.r.t.\ a 1-cutset, a proper 1-join or a proper 2-cutset.  Then
$G\in {\cal C}$ if and only if $G_X \in {\cal C}$ and $G_Y \in {\cal
  C}$.
\end{lemma}

\begin{proof}
Suppose $G_X$ and $G_Y$ are the blocks of decomposition of $G$
w.r.t.\ a 1-cutset or a proper 1-join. Then $G_X$ and $G_Y$ are
induced subgraphs of~$G$, and hence if $G \in {\cal C}$ then $G_X \in
{\cal C}$ and $G_Y \in {\cal C}$.  Conversely, suppose that $G_X \in
{\cal C}$ and $G_Y \in {\cal C}$. If they are blocks w.r.t.\ a
1-cutset, then every cycle of $G$ belongs to $G_X$ or $G_Y$, and hence
$G \in {\cal C}$.  Assume they are blocks w.r.t.\ a proper
1-join. Then every cycle $C$ that has at least two nodes in $X$ and at
least two nodes in $Y$, has at least two nodes in $A$ and at least two
nodes in~$B$. Since $A$ and $B$ are stable sets, $C$ is either a
square or it has at least two chords.  It follows that every cycle of
$G$ with a unique chord is contained in $G_X$ or $G_Y$ (where possibly
the marker node plays the role of one of the nodes of the
cycle). Hence $G \in {\cal C}$.

Now suppose that $G_X$ and $G_Y$ are the blocks of decomposition of
$G$ w.r.t.\ a proper 2-cutset, with split $(X,Y,a,b)$.  Suppose $G \in
{\cal C}$.  Suppose w.l.o.g.\ that $G_X$ contains a cycle $C$ with a
unique chord. Then $C$ must contain $c$.  Let $P$ be an ab-path in
$G[Y \cup \{ a,b \}]$.  Then $G[(V(C)\setminus \{c\}) \cup V(P)]$ is a
cycle with a unique chord, a contradiction. So $G_X \in {\cal C}$ and
$G_Y \in {\cal C}$.

To prove the converse, assume that
  $G_X \in {\cal C}$ and $G_Y \in {\cal C}$, and $G$ contains a cycle
$C$ with a unique chord. Since $C$ cannot be contained in $G_X$ nor 
$G_Y$, it must contain a node of $X$ and a node of $Y$, and hence it
contains $a$ and $b$. Let $P_X$ (resp. $P_Y$) be the section of $C$ in
$G[X \cup \{ a,b \} ]$ (resp. $G[Y \cup \{ a,b \} ]$). 
Since $C$ contains a unique
chord, w.l.o.g.\ $P_Y$ is a path and $P_X$ has a unique chord. 
But then $G_X[V(P_X) \cup \{ c\} ]$ is a cycle with a
unique chord, a contradiction. 
\end{proof}

Theorem \ref{th:1} and Lemma \ref{iff0} actually give us a complete
structure theorem for class ${\cal C}$, i.e. every graph in ${\cal C}$
can be built starting from basic graphs, that can be explicitly constructed,
and gluing them together by prescribed composition operations, and all
graphs built this way are in ${\cal C}$. 

Cliques, holes and induced subgraphs of the Petersen graph or the
Heawood graph can clearly be explicitly constructed.  Also strongly
2-bipartite graphs can be constructed as follows. Let $X$ and $Y$ be
node sets.  We construct a bipartite graphs with bipartition $(X,Y)$
by making every node of $X$ adjacent to two nodes of $Y$.  Every
strongly 2-bipartite graph can be constructed this way, and every
graph constructed this way belongs to ${\cal C}$.  Indeed, a graph so
constructed does not have an edge both of whose endnodes are of degree
at least 3, whereas a chord of a cycle has endnodes that are both of
degree at least 3.

The composition operations we need are just the reverse of our
decompositions, and the union of two graphs. Each operation takes as
input two node disjoint graphs $G_1$ and $G_2$, and outputs a third
graph $G$.

\begin{description}
\item[\sc Operation ${\cal O}_0$] is the operation of taking the disjoint union
of two graphs, i.e. $V(G)=V(G_1) \cup V(G_2)$ and $E(G)=E(G_1) \cup E(G_2)$.

\item[\sc Operation ${\cal O}_1$] is the operation that is the reverse of
1-cutset decomposition. For some node $u$ of $G_1$ and some node $w$ of $G_2$,
$G$ is obtained from the disjoint union of $G_1 \setminus \{ u \}$
and $G_2 \setminus \{ w \}$, by adding a new node $v$ and all the edges 
between $v$ and $N_{G_1}(u) \cup N_{G_2}(w)$.

\item[\sc Operation ${\cal O}_2$] is the operation that is the reverse of
  proper 1-join decomposition.  For some node $u$ (resp. $v$) of $G_1$
  (resp. $G_2$) such that $N_{G_1}(u)$ (resp. $N_{G_2}(v)$) is a
  stable set of size at least 2, $G$ is obtained from the disjoint
  union of $G_1 \setminus \{ u \}$ and $G_2 \setminus \{ v \}$ by
  adding all edges between $N_{G_1}(u)$ and $N_{G_2}(v)$.

\item[\sc Operation ${\cal O}_3$] is the operation that is the reverse
  of proper 2-cutset decomposition.  For some degree 2 node $u$
  (resp. $v$) of $G_1$ (resp. $G_2$) which is not a 1-cutset, with
  neighbors $u_1$ and $u_2$ (resp. $v_1$ and $v_2$) such that $u_1$
  and $u_2$ (resp. $v_1$ and $v_2$) are nonadjacent, and
  $(d_{G_1}(u_1)-1)+(d_{G_2}(v_1)-1)\geq 3$ and
  $(d_{G_1}(u_2)-1)+(d_{G_2}(v_2)-1)\geq 3$, $G$ is obtained from the
  disjoint union of $G_1 \setminus \{ u,u_1,u_2 \}$ and $G_2 \setminus
  \{ v,v_1,v_2 \}$ by adding new nodes $w_1$ and $w_2$ and all edges
  between $w_1$ and $(N_{G_1}(u_1) \setminus \{ u \} ) \cup
  (N_{G_2}(v_1) \setminus \{ v \} )$ and between $w_2$ and
  $(N_{G_1}(u_2) \setminus \{ u \} ) \cup (N_{G_2}(v_2) \setminus \{ v
  \} )$.
\end{description}

\begin{theorem}
  \label{th:structure}
  If $G \in {\cal C}$ then either $G$ is basic or can be obtained
  starting from basic graphs by repeated applications of operations
  ${\cal O}_0, \ldots ,{\cal O}_3$. Conversely, every graph obtained
  in this way is in ${\cal C}$.
\end{theorem}
\begin{proof}
Follows from Theorem \ref{th:1} and Lemma \ref{iff0}.
\end{proof}

\section{Constructing a decomposition tree}
\label{sec:decomp}

We will now construct a decomposition tree for an input graph $G$, and
then use this tree to obtain an ${\cal O} (nm)$ recognition algorithm
for class ${\cal C}$ (in descriptions of algorithms, $n$ stands for
the number of nodes and $m$ for the number of edges).  An ${\cal
  O}(n^5)$ or a slightly more involved ${\cal O}(n^4)$ algorithm could
be obtained from first principles, but we use sophisticated algorithms
from other authors, namely Dahlhaus~\cite{dahlhaus:split}, Hopcroft
and Tarjan~\cite{hopcroft.tarjan:447,tarjan:dfs,hopcroft.tarjan:3con}
to get our algorithm to run in ${\cal O}(nm)$-time.  Note that we do not
use the full strength of the works of these authors since they are
able to decompose fully a graph in linear time using 1-joins, or using
1-cutsets, or using 2-cutsets.  But their notions of decompositions
differ slightly from what we need so we just use their algorithms to
find the cutsets in linear time.  We leave as an open question whether
it is possible to recognize graphs in ${\cal C}$ in ${\cal
  O}(n+m)$-time.

We could use the definition of blocks of decomposition from
Section~\ref{sec:st} to construct a decomposition tree, and use it to
obtain the recognition algorithm, but such a tree cannot be used for
our coloring algorithm, because for coloring we need the blocks of a
square-free graph with respect to a proper 2-cutset to be also
square-free.  So in this section, blocks of decomposition w.r.t.\ a
1-cutset and a proper 1-join stay the same as in Section~\ref{sec:st}
but the blocks of decomposition w.r.t.\ a proper 2-cutset are
redefined here below.

The \emph{block} $G_X$ (resp. $G_Y$) of a graph $G$ with respect to a
proper 2-cutset with split $(X, Y, a, b)$ is the graph obtained as
follows:  
\begin{itemize}
\item
  if there exists a node $c$ of $G$ such that $N(c) = \{a, b\}$, then
  take such a node $c$, and let $G_X = G[X \cup \{a,b,c\}]$ and $G_Y =
  G[Y \cup \{a,b,c\}]$;
\item
  else $G_X$ (resp. $G_Y$) is the block defined in Section~\ref{sec:st}
  that is the graph obtained by taking $G[X\cup \{ a,b\}]$
  (resp. $G[Y\cup \{ a,b\}]$) and adding a node $c$ adjacent to $a,
  b$.
\end{itemize}

Node $c$ is called the \emph{marker} of the block $G_X$ (resp. $G_Y$).

\begin{lemma}
  \label{l:Htriangle}
  Let $G \in {\cal C}$ and suppose that $G_X$ and $G_Y$ are the blocks
  of decomposition of $G$ w.r.t.\ a 1-cutset, a proper 1-join or a
  proper 2-cutset. If $G$ is connected and triangle-free then $G_X$
  and $G_Y$ are connected and triangle-free.
\end{lemma}

\begin{proof}
  The blocks of $G$ are clearly connected.  The blocks of $G$ with
  respect to a 1-cutset or a proper 1-join are induced subgraphs of
  $G$ so they are triangle-free.  If one of the block of $G$ w.r.t.\ a
  proper 2-cutset $\{a, b\}$ contains a triangle then this triangle
  must contain the marker $c$.  So the triangle must be $abc$ and this
  contradicts $ab \notin E(G)$.
\end{proof}

\begin{lemma}\label{l0}
  Let $G \in {\cal C}$ and suppose that $G_X$ and $G_Y$ are the blocks
  of decomposition of $G$ w.r.t.\ a proper 2-cutset with split $(X, Y,
  a, b)$. If $G$ is connected, triangle-free, square-free,
  Petersen-free, has no 1-cutset and no proper 1-join, then $G_X$ and
  $G_Y$ have the same property.
\end{lemma}

\begin{proof}
  For connectivity and triangles, the lemma follows from
  Lemma~\ref{l:Htriangle}.

  For squares, suppose that w.l.o.g.\ $G_X$ contains a square $C$.
  Since $G$ is square-free, $C$ contains the marker node $c$ (that is
  not a real node of $G$), and hence $C=cazbc$, for some node $z \in
  X$. Since $c$ is not a real node of $G$, $d_G(z) >2$ for otherwise,
  $z$ would have been chosen to serve as a marker.  Let $z'$ be a
  neighbor of $z$ that is distinct from $a$ and $b$.  Note that since
  $G$ is triangle-free, $az'$ and $bz'$ are not edges.  Since $z$ is
  not a 1-cutset, there exists a path $P$ in $G[X \cup \{a, b\}]$ from
  $z'$ to $\{a, b\}$. We choose $z'$ and $P$ subject to the minimality
  of $P$. So, w.l.o.g.\ $z'Pa$ is a path. Note that $b$ is not
  adjacent to the neighbor of $a$ along $P$ since $z$ is the unique
  common neighbor of $a, b$ because $G$ is square-free.  So by
  minimality of $P$, $b$ does not have a neighbor in $P$.  Now let $Q$
  be a path from $a$ to $b$ whose interior is in $Y$.  So, $bzz'PaQb$
  is a cycle with a unique chord (namely $az$), a contradiction.

  For the Petersen graph, it suffices to notice that if a block of $G$
  contains it, then the marker $c$ must be in it, and this is a
  contradiction since $c$ is of degree two.

  For 1-cutsets, suppose w.l.o.g.\ that $G_X$ has a 1-cutset with split
  $(A,B,v)$.  Since $G$ is connected and $G[X \cup \{ a,b \}]$
  contains an $ab$-path, $v \neq c$ (where $c$ is the marker node of
  $G_X$). Suppose $v=a$. Then w.l.o.g.\ $b \in B$, and hence $(A,B \cup
  Y,a)$ is a split of a 1-cutset of~$G$ (with possibly $c$ removed
  from $B \cup Y$, if $c$ is not a real node of $G$), a
  contradiction. So $v \neq a$ and by symmetry $v \neq b$. So $v \in X
  \setminus \{ c \}$. W.l.o.g.\ $\{ a,b,c \} \subseteq B$. Then $(A,B
  \cup Y,v)$ is a split of a 1-cutset of $G$ (with possibly $c$
  removed from $B \cup Y$, if $c$ is not a real node of $G$), a
  contradiction. 

  For proper 1-joins, it suffices to notice that the blocks of $G$ are
  square-free, so they cannot have a proper 1-join. 
\end{proof}

\begin{lemma}\label{iff}
  Let $G_X$ and $G_Y$ be the blocks of decomposition of $G$ w.r.t.\ a
  1-cutset, a proper 1-join or a proper 2-cutset.  Then $G\in {\cal
    C}$ if and only if $G_X \in {\cal C}$ and $G_Y \in {\cal C}$.
\end{lemma}

\begin{proof}
  By Lemma~\ref{iff0} we may assume that $G_X$ and $G_Y$ are the
  blocks of decomposition of $G$ w.r.t.\ a proper 2-cutset, with split
  $(X,Y,a,b)$ and that the marker $c$ is a real node of $G$.  So $G_X$
  and $G_Y$ are induced subgraphs of~$G$.  Hence, $G \in {\cal C}$
  implies $G_X \in {\cal C}$ and $G_Y \in {\cal C}$.

  To prove the converse, assume that $G_X \in {\cal C}$ and $G_Y \in
  {\cal C}$, and $G$ contains a cycle $C$ with a unique chord. Since
  $C$ cannot be contained in $G_X$ nor $G_Y$, it must contain a node
  of $X$ and a node of $Y$, and hence it contains $a$ and $b$. Let
  $P_X$ (resp. $P_Y$) be the section of $C$ in $G[X \cup \{ a,b \}]$
  (resp. $G[Y \cup \{ a,b \}]$).  Since $C$ contains a unique chord,
  w.l.o.g.\ $P_Y$ is a path and $P_X$ has a unique chord. Note that $c
  \not \in V(P_X)$, since $c$ is of degree 2 in $G$ and it is adjacent
  to both $a$ and $b$. Hence $G_X[V(P_X) \cup \{ c\} ]$ is a cycle
  with a unique chord, a contradiction.
\end{proof}

An algorithm of Hopcroft and
Tarjan~\cite{hopcroft.tarjan:447,tarjan:dfs} finds in linear time a
1-cutset of $G$ (if any).  An algorithm of
Dahlhaus~\cite{dahlhaus:split} finds in linear time a 1-join of $G$ if
any.  The next lemma shows how to use this algorithm to find a proper
1-join or determine that $G \not \in {\cal C}$.

\begin{lemma}\label{p1j}
  Let $G$ be a graph that is not a clique and has no 1-cutset.  Assume
  $G$ has a 1-join. If this 1-join is not proper, then $G \not \in
  {\cal C}$.
\end{lemma}

\begin{proof}
  Let $(X,Y,A,B)$ be the split of a 1-join of $G$ that is not proper.
  If $|A|=1$ then $A$ is a 1-cutset of $G$, a contradiction.  So
  $|A|\geq 2$, and by symmetry $|B| \geq 2$. Since the 1-join is not
  proper, w.l.o.g.\ there is an edge with both ends in $A$. This edge
  together with any node of $B$ forms a triangle, and so by Theorem
  \ref{th:triangle} (and since $G$ is not a clique and has no
  1-cutset) $G \not \in {\cal C}$.
\end{proof}

Recall that in a graph $G$ two nodes $a$ and $b$ form a {\em 2-cutset}
if $G\setminus \{ a,b \}$ is disconnected.  Hopcroft and
Tarjan~\cite{hopcroft.tarjan:3con} give an algorithm that finds a
2-cutset in a graph (if any) in linear time.  This 2-cutset is not
necessarily a proper 2-cutset (which is what we need).  We now show
how to find a proper 2-cutset in linear time.

Recall that if $H$ is an induced subgraph of $G$ and $D$ is a set
of nodes of $G \setminus H$, the {\em attachment of $D$ over $H$}
is the set of all nodes of $H$ that have a neighbor in $D$.

\begin{lemma}\label{p2calg}
  There is an algorithm with the following specifications.
  \begin{description}
  \item[\rule{1em}{0ex}\sc Input:] A connected graph $G$ that has no 1-cutset nor a
    proper 1-join, and is not basic.
  \item[\rule{1em}{0ex}\sc Output:] 
    $G$ is correctly identified as not belonging
    to ${\cal C}$, or a proper 2-cutset of $G$.
  \item[\rule{1em}{0ex}\sc Running time:]  
    ${\cal O}(n+m)$.
  \end{description}
\end{lemma}

\begin{proof}
  Consider the following algorithm.

  \begin{description}
  \item[\rule{1em}{0ex}\sc Step 1:] Let $G_2$ be the subgraph of~$G$
    induced by the degree 2 nodes of $G$.  Since $G$ is connected, has
    no 1-cutset and is not a chordless cycle (because $C_k$, $3\leq
    k\neq 4$, is basic and $C_4$ admits a proper 1-join), the
    connected components of $G_2$ are paths, and for every such path
    $P$, the attachment of $P$ over $G\setminus P$ consists of two
    distinct nodes of $G$ that are both of degree at least 3 in $G$.
    If there exists a path $P$ in $G_2$ whose attachment $\{ a,b \}$
    over $G \setminus P$ is such that $ab$ is an edge, then output $G
    \not \in {\cal C}$ and stop.

  \item[\rule{1em}{0ex}\sc Step 2:] If there is a path $P$ in $G_2$ of
    length at least~1 then let $\{a, b\}$ be the attachment of $P$
    over $G \setminus P$.  Output $\{ a,b \}$ as a proper 2-cutset of
    $G$ and stop.
 
  \item[\rule{1em}{0ex}\sc Step 3:] Now all paths of $G_2$ are of
    length~0.  Create the graph $G'$ from $G\setminus V(G_2)$ as
    follows: for every path $P$ of $G_2$ put an edge between the pair
    of nodes that are the attachment of $P$ over $G\setminus P$.  Note
    that if $G_2$ is empty, then $G=G'$.  If $G'$ has no 2-cutset,
    output $G\notin {\cal C}$ and stop.

  \item[\rule{1em}{0ex}\sc Step 4:] Find a 2-cutset $\{ a,b \}$ of
    $G'$. Note that $\{ a,b \}$ is also a 2-cutset of $G$. If $ab$ is
    an edge of $G$, then output $G \not \in {\cal C}$ and
    stop. Otherwise, output $\{a, b\}$ as a proper 2-cutset of $G$ and
    stop.
\end{description}

Since 2-cutsets in Step~3 and~4 can be found in time ${\cal O}(n+m)$
by the Hopcroft and Tarjan algorithm \cite{hopcroft.tarjan:3con}, it
is clear that the above algorithm can be implemented to run in time
${\cal O} (n+m)$.  We now prove the correctness of the algorithm.

First note that since $G$ is not a clique and it does not have a
1-cutset, all nodes of $G$ have degree at least 2.  Suppose the
algorithm stops in Step~1. So there exists a path $P$ in $G_2$ whose
attachment over $G \setminus P$ induces an edge $ab$.  Since $d_G (a)
\geq 3$, it follows that $V(G) \setminus (V(P) \cup \{ a,b \} ) \neq
\emptyset$, and hence $\{ a,b \}$ is a 2-cutset of $G$. So by Lemma
\ref{abedge}, the algorithm correctly identifies $G$ as not belonging
to ${\cal C}$.

Suppose the algorithm stops in Step~2.  By Step~1, $ab$ is not an
edge. Since $d_G (a) \geq 3$, $|V(G) \setminus (V(P) \cup \{ a,b \} )|
\geq 2$, and since $P$ is of length at least 1, $|V(P)| \geq 2$.
Since $G$ has no 1-cutset, there is an $ab$-path in $G \setminus P$.
Hence $\{ a,b \}$ is a proper 2-cutset of $G$.

Suppose the algorithm stops in Step~3.  This means that $G'$ has no
2-cutset.  Since the output is $G\notin {\cal C}$, the only problem is
when $G\in {\cal C}$, so let us suppose for a contradiction $G\in
{\cal C}$.  Then by Theorem~\ref{th:1}, $G$ has a proper 2-cutset $\{
a,b \}$ with split $(X,Y,a,b)$.  Since $d_G (a) \geq 3$ and $d_G (b)
\geq 3$, $\{ a,b \} \subseteq V(G')$. If $|X \cap V(G')| \geq 1$ and
$|Y \cap V(G')| \geq 1$, then $\{ a,b \}$ is a 2-cutset of $G'$, so we
may assume w.l.o.g.\ that $X \cap V(G')=\emptyset$.  Since $\{ a,b \}$
is a proper 2-cutset of $G$, $|X|\geq 2$.  So $X$ contains two nodes
$u_1$ and $u_2$ that are both of degree 2 in $G$. By Step~2, $u_1$ and
$u_2$ are paths of $G_2$ of length~0. Since $\{ a,b \}$ is a cutset of
$G$, and $G$ is connected and has no 1-cutset, it follows that both
$u_1$ and $u_2$ are adjacent to both $a$ and $b$. So $au_1bu_2$ is a
square of $G$, so by Theorem~\ref{th:square}, $G$ must have a proper
1-join, a 1-cutset or must be basic, in either case a contradiction.

Suppose the algorithm stops in Step~4.  Let $\{ a,b \}$ be a 2-cutset
of $G'$. Then clearly $\{ a,b \}$ is also a 2-cutset of~$G$. If $ab$
is an edge of $G$, then by Lemma \ref{abedge} the algorithm correctly
identifies $G$ as not belonging to ${\cal C}$.  So assume $ab$ is not
an edge of $G$. Note that for every $u \in V(G')$, $d_G (u) \geq
3$. In particular, since $a,b \in V(G')$, $d_G (a)\geq 3$ and $d_G (b)
\geq 3$. Let $C'$ be a connected component of $G' \setminus \{ a,b
\}$. Let $u$ be a node of $C'$, and let $C$ be the connected component
of $G \setminus \{ a,b \}$ that contains $u$.  Since $d_G (u) \geq 3$,
it follows that $|V(C)| \geq 2$.  This is true of every connected
component of $G' \setminus \{ a,b \}$.  Also, since $G$ is connected
and has no 1-cutset, for every connected component $C$ of $G \setminus
\{ a,b \}$ there is an $ab$-path in $G[C\cup \{ a,b \} ]$.  Therefore
$\{ a,b \}$ is a proper 2-cutset of $G$.
\end{proof}

A \emph{decomposition tree} of a graph $G$ is a rooted tree $T_G$
such that the following hold:

\begin{enumerate}
\item 
  $G$ is the root of $T_G$.
\item 
  For every non-leaf node $H$ of $T_G$, the children of $H$ are the
  blocks of decomposition of $H$ w.r.t.\ a 1-cutset, a proper 1-join
  or a proper 2-cutset of $H$.
\end{enumerate}

\begin{lemma}
  \label{l:sizeT}
  Let $G$ be any graph and let $T$ be a decomposition tree of $G$.
  Then $T$ has size ${\cal O}(n)$.
\end{lemma}

\begin{proof}
  Note that $T$ is finite since the children of a graph are smaller
  than its parent.  Let $T'$ be the subtree of $T$ on the nodes that
  are graphs on at least five nodes.  For any graph $G$ we define
  $\varphi(G) = |V(G)| - 4$. It is easily seen that when $G_{X}$,
  $G_{Y}$ are the blocks of $G$ with respect to some decomposition,
  then $\varphi(G) \geq \varphi(G_{X}) + \varphi(G_{Y})$.  Indeed, for
  a 2-cutset with split $(X, Y, a, b)$ where the marker $c$ is not a
  real node of $G$ the inequality follows from $\varphi(G) =
  |X|+|Y|-2$, $\varphi(G_X) = |X|-1$ and $\varphi(G_Y) = |Y|-1$.  For
  the other decompositions, the proof is similar.

  Since in $T'$ every node is a graph on at least five nodes, every
  node $F$ of $T'$ is such that $\varphi(F) \geq 1$. So the number of
  leaves of $T'$ is at most $\varphi(G)$. Hence the size of $T'$ is
  ${\cal O} (n)$.  It follows that the size of $T$ is also ${\cal O}
  (n)$, since the decomposition of the graphs that have fewer than 5
  nodes is bounded by a constant.
\end{proof}

Decomposition trees would be sufficient for a recognition algorithm,
but for coloring we need a more sophisticated kind of tree.  A
\emph{proper decomposition tree} of a connected graph $G\in {\cal C}$
is a rooted tree $T_G$ such that the following hold:

\begin{enumerate}
\item 
  $G$ is the root of $T_G$.
\item
  Every node of $T_G$ is a connected graph.
\item 
  Every leaf of $T_G$ is basic.
\item 
  Every non-leaf node $H$ of $T_G$ is of one of the following type:
  
  {\sc Type 1:} the children of $H$ in $T_G$ are the blocks of
  decomposition w.r.t.\ a 1-cutset or a proper 1-join;

  {\sc Type 2:} $H$ and all its descendants are Petersen-, triangle-,
  square-free and have no 1-cutset and no proper 1-join.  Moreover the
  children of $H$ in $T_G$ are the blocks of decomposition w.r.t.\ a
  proper 2-cutset and every non-leaf descendant of $H$ is of type~2.
\item
  If a node of $T_G$ is a triangle-free graph then all its descendants
  are triangle-free graphs.  
\end{enumerate}

\begin{theorem}\label{dtalg}
  There is an algorithm with the following specifications.
  \begin{description}
  \item[\rule{1em}{0ex}\sc Input:] A connected graph $G$.
  \item[\rule{1em}{0ex}\sc Output:] $G$ is correctly identified as not
    belonging to ${\cal C}$, or if $G \in {\cal C}$, a proper decomposition
    tree for $G$.
  \item[\rule{1em}{0ex}\sc Running time:] ${\cal O}(mn)$.
  \end{description}
\end{theorem}

\begin{proof}
  Consider the following algorithm.

  \begin{description}
  \item[\rule{1em}{0ex}\sc Step~1:] Let $G$ be the root of $T_G$. 

  \item[\rule{1em}{0ex}\sc Step~2:] If all the leaves of $T_G$ have
    been declared as LEAF NODE, then output $T_G$ and stop. Otherwise,
    let $H$ be a leaf of $T_G$ that has not been declared a LEAF NODE.

  \item[\rule{1em}{0ex}\sc Step~3:] If $H$ is basic, declare $H$ to be
    a LEAF NODE and go to Step~2.

  \item[\rule{1em}{0ex}\sc Step~4:] If $H$ has a 1-cutset, then let
    the children of $H$ be the blocks of decomposition w.r.t.\ this
    1-cutset, and go to Step~2.
 
  \item[\rule{1em}{0ex}\sc Step~5:] If $H$ has a 1-join, then check
    whether this 1-join is proper. If it is, then let the children of
    $H$ be the blocks of decomposition by this proper 1-join, and go
    to Step~2.  If it is not, then output $G \not \in {\cal C}$ and
    stop.
 
  \item[\rule{1em}{0ex}\sc Step~6:] Apply algorithm from
    Lemma~\ref{p2calg} to $H$.  If the output is that $G \not \in
    {\cal C}$ then output the same and stop.  Otherwise a proper
    2-cutset is found.  Then let the children of $H$ be the blocks of
    decomposition by this proper 2-cutset, and go to Step~2.
  \end{description}

  Note that this algorithm stops, because the children of a graph are
  smaller than its parent.  We first prove the correctness of the
  algorithm.  If the algorithm stops in Step~5, then by
  Lemma~\ref{p1j}, $G$ is correctly identified as not belonging to
  ${\cal C}$.  If the algorithm stops in Step~6, then by
  Lemma~\ref{p2calg}, $G$ is correctly identified as not belonging to
  ${\cal C}$.  So we may assume that the algorithm stops in Step~2.
  This means that the algorithm outputs a decomposition tree~$T_G$.
  By Lemma~\ref{iff}, it follows that $G\in {\cal C}$ since every leaf
  of $T_G$ is basic.  Let us check that $T_G$ is proper.

  Clearly, $G$ is the root of $T_G$ and every leaf of $T_G$ is basic.
  Since $G$ is connected, and by the construction of blocks of
  decomposition, all nodes of $T_G$ are connected graphs.  Let $H$ be
  a non-leaf node of $T_G$.  Note that $H$ is not basic because of
  Step~3.
 
  If $H$ is Petersen-, triangle- and square-free, has no 1-cutset and
  no proper 1-join, then by Theorem~\ref{th:nosquare}, $H$ has a
  proper 2-cutset and is decomposed along a proper 2-cutset because of
  Step~6.  Also, by Lemma~\ref{l0} the children of~$H$ are also
  connected, Petersen-, triangle-, square-free, and have no 1-cutset
  and no proper 1-join.  So by induction, every non-leaf descendant of
  $H$ is decomposed along proper 2-cutsets and $H$ is of type~2.

  Else, $H$ contains a triangle, a square, the Petersen graph or has a
  1-cutset or a proper 1-join.  By Theorems~\ref{th:triangle} and
  \ref{th:square}, $H$ must have a 1-cutset or a proper 1-join.  Note
  that this 1-cutset or proper 1-join is discovered by the algorithm
  rather than a possible proper 2-cutset.  So, $H$ is of type~1.

  So every non-leaf node of $T_G$ is of type 1 or 2, and by
  Lemma~\ref{l:Htriangle}, if a node of $T_G$ is a triangle-free graph
  then all its descendants are triangle-free graphs.  We have proved
  that $T_G$ is a proper decomposition tree.

  We now show that the algorithm can be implemented to run in time
  ${\cal O} (nm)$.  Testing whether a graph is a clique in Step~3
  relies only on a check of the degrees: $H$ is a clique if and only
  if every node has degree $n-1$, so this can be done in time ${\cal
    O} (n+m)$.  To decide whether a graph is strongly 2-bipartite, we
  also check the degrees to be sure that nodes of degree~2 and nodes
  of degree at least~3 form stable sets.  We still have to check that
  $H$ is square-free, but this can be done by running the ${\cal O}
  (n+m)$ algorithm of Dahlhaus~\cite{dahlhaus:split} for 1-joins
  because at this step, $H$ contains a square if and only if $H$ has a
  1-join.

  To find a 1-cutset in Step~4, we use the ${\cal O} (n+m)$ algorithm
  of Hopcroft and Tarjan~\cite{hopcroft.tarjan:447,tarjan:dfs}.  To
  find a 1-join in Step~5, we use the ${\cal O} (n+m)$ algorithm of
  Dahlhaus~\cite{dahlhaus:split}.  By Lemma~\ref{p2calg}, Step~6 can
  be implemented to run in time ${\cal O} (n+m)$.  Now we note that
  when the algorithm stops, it has computed a decomposition tree (that
  will be output or not when $G\notin {\cal C}$), and the numbers of
  steps processed by the algorithm is bounded by the size of this
  tree.  By Lemma~\ref{l:sizeT} the size of the tree is ${\cal O}
  (n)$, so we have to run ${\cal O} (n)$ times each of the steps, and
  hence the overall complexity is ${\cal O} (nm)$.
\end{proof}

\begin{theorem}
  There exists an ${\cal O}(nm)$-time algorithm that decides whether a
  graph is in ${\cal C}$.
\end{theorem}

\begin{proof}
  Apply the ${\cal O} (nm)$ algorithm from Theorem \ref{dtalg}.  If
  the output is $G \not \in {\cal C}$ then output the same.  Else $G$
  has a proper decomposition tree and $G\in {\cal C}$ by
  Lemma~\ref{iff}.
\end{proof}

\section{Coloring}
\label{s:bounding}

Let us call \emph{third color} of a graph any stable set that contains
at least one node of every odd cycle. Any graph that admits a third
color $S$ is 3-colorable: give color~3 to the third color; since $G\sm
S$ contains no odd cycle, it is bipartite: color it with colors~1,
2. We shall prove by induction that any triangle-free graph in $\cal
C$ has a third color. But for the sake of induction, we need to prove
a stronger statement.

Let us call \emph{strong third color} of a graph any stable set that
contains at least one node of every cycle (odd or even).  By $N[v]$ we
denote $\{v\} \cup N(v)$. When $v$ is a node of a graph $G$, a pair of
disjoint subsets $(R, T)$ of $V(G)$ is \emph{admissible with respect
  to $G$ and $v$} if one of the following holds (see
Fig.~\ref{f:admis}):

\begin{itemize}
\item
  $T = N(v)$ and $R = \{v\}$;
\item
  $T = \emptyset$ and $R = N[v]$;
\item
  $v$ is of degree two, $N(v) = \{u, w\}$, $T = \{u\}$, $R = \{v,
  w\}$;
\item
  $v$ is of degree two, $N(v) = \{u, w\}$, $T = \{u\}$, $R = N[w]$;
\item
  $v$ is of degree two, $N(v) = \{u, w\}$, $T = \emptyset$, $R = \{u\}
  \cup N[w]$.
\end{itemize}

\begin{figure}
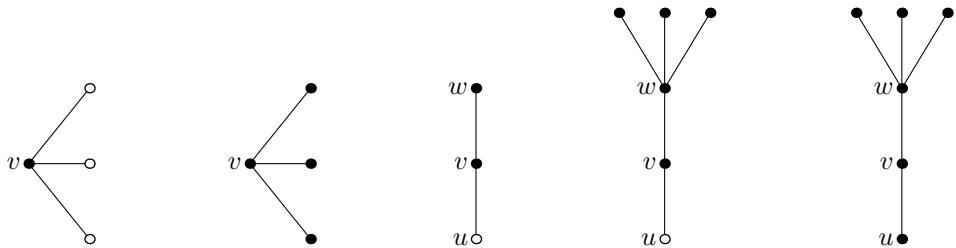

\center
\includegraphics{figChord.11}\rule{4.1em}{0ex}
\includegraphics{figChord.12}\rule{4.1em}{0ex}
\includegraphics{figChord.13}\rule{4.1em}{0ex}
\includegraphics{figChord.14}\rule{4.1em}{0ex}
\includegraphics{figChord.15}
\caption{Five examples of admissible pairs (nodes of $T$ are white,
  nodes of $R$ are black)\label{f:admis}}
\end{figure}

We say that a pair of disjoint subsets $(R,T)$ is an {\em admissible
  pair} of $G$ if for some $v \in V(G)$, $(R,T)$ is admissible
w.r.t.\ $G,v$.  An admissible pair $(R, T)$ should be seen as a
constraint for coloring: we will look for third colors (sometimes
strong, sometimes not) that must contain every node of $T$ and no node
of $R$.  We will do this first in basic graphs, and then by induction
in all triangle-free graphs of $\cal C$, thus proving that they are
3-colorable.

\begin{lemma}
  \label{th:tcbasic}
  Let $G$ be a triangle-free basic graph that is not the Petersen
  graph.  Let $(R, T)$ be an admissible pair of~$G$.  Then $G$ admits
  a strong third color $S$ such that $T \subseteq S$ and $S \cap R =
  \emptyset$.  Furthermore, $S$ can be found in time ${\cal O}(n+m)$.
\end{lemma}

\begin{proof}
Note that squares may fail to admit strong third color that satisfies
our constraints (because when $G$ is a square $R=V(G)$ is possible).
But squares are not basic.  The proof follows from the following
claims, since it will be clear that all $S$'s found in them can be
found in time ${\cal O} (n+m)$.

  \begin{claim}\label{bc0}
    The lemma holds when $G$ is a chordless cycle of length at least~7.
  \end{claim}

  \begin{proofclaim}
    Because then any non-empty set of nodes is a strong third color.
    Since $R=V(G)$ is impossible because any path in $G[R]$ is of
    length at most $3$, it is always possible to pick a node of $G$
    not in $R$. 
  \end{proofclaim}

  \begin{claim}\label{bc1}
    The lemma holds when $G$ is a clique, a strongly 2-bipartite graph
    or is an induced subgraph of the Heawood graph.
  \end{claim}

  \begin{proofclaim}
    Note that $G$ is bipartite (for cliques, because it is triangle
    free).  Let $A, B$ be a bipartition of $G$.  Note that $A, B$ can
    be computed in linear time.  Up to symmetry between $A, B$ we may
    assume $T \subset A$ and $|A \cap R| \leq 2$. Let $S = A \sm
    R$. So $T \subset S$ and $S \cap R = \emptyset$. Moreover, $|A \sm
    S| \leq 2$. So every cycle in $G\sm S$ contains at most two nodes
    of $A$, and since $G$ is square-free, there is no such cycle.
  \end{proofclaim}

  \begin{claim}\label{pet3}
    The lemma holds when $G$ is a proper induced subgraph of the
    Petersen graph.
  \end{claim}

  \begin{proofclaim}
    Note that by assumption, $G$ is not the Petersen graph.  We use
    our notation for the Petersen graph $\Pi$. So $V(G) \subsetneq
    V(\Pi) = \{a_1, \dots, a_5, b_1, \dots, b_5\}$. Let $v$ be a node
    of $G$ and $(R, T)$ be admissible with respect to $G, v$.  We may
    assume $v = a_1$ since the Petersen graph is
    vertex-transitive. Note that $a_1 \in V(G)$.

    Suppose $T = N(a_1)$ and $R = \{a_1\}$.  Then we put $Q = \{a_2,
    a_5, b_1\}$ and we observe that $\Pi \sm Q$ is a $C_6$ plus an
    isolated node.  But some node $z\neq a_1$ of $\Pi$ is not a node
    of $G$.  If $z$ is in the $C_6$, then $S = Q \cap V(G)$ is a
    strong third color of $G$ such that $T \subseteq S$ and $S \cap R
    = \emptyset$. Else $z$ must be a neighbor of $a_1$, say $a_5$ up
    to symmetry. So, $S = (Q \cup \{a_4\}) \cap V(G)$ is a strong
    third color of $G$ such that $T \subseteq S$ and $S \cap R =
    \emptyset$.

    Suppose $T = \emptyset$ and $R = N[v]$.  Then we put $Q = \{a_3,
    b_3, b_5\}$ and we observe that $\Pi \sm Q$ is a tree.  So $S = Q
    \cap V(G)$ is a strong third color of $G$ such that $T \subseteq
    S$ and $S \cap R = \emptyset$.

    From here on we may assume that $v$ is a node of degree two of
    $G$.  So up to symmetry we may assume $b_1 \notin V(G)$ and
    $N_G(v) \subseteq \{a_2, a_5\}$.
    
    Suppose $N(v) = \{u, w\}$, $T = \{u\}$, $R = \{v, w\}$. So
    w.l.o.g.\ $T = \{a_2\}$ and $R = \{a_1, a_5\}$. Then we put $Q =
    \{a_2, b_3\}$ and we observe that $\Pi \sm Q$ is a tree.  So $S =
    Q \cap V(G)$ is a strong third color of $G$ such that $T \subseteq
    S$ and $S \cap R = \emptyset$.

    Suppose $N(v) = \{u, w\}$, $T = \{u\}$, $R = N[w]$.
    So up to symmetry we may assume $T = \{a_2\}$, $R = \{a_1, a_5,
    a_4, b_3\} \cap V(G)$. Then we put $Q = \{a_2, b_2, b_5\}$ and we
    observe that like in the previous case that $\Pi \sm Q$ is a tree.
    So $S = Q \cap V(G)$ is a strong third color of $G$ such that $T
    \subseteq S$ and $S \cap R = \emptyset$.

    Suppose $N(v) = \{u, w\}$, $T = \emptyset$, $R = \{u\} \cup
    N[w]$. So up to symmetry we may assume $R = \{a_1, a_2, a_5, b_3,
    a_4\}\cap V(G)$. We put $Q = \{b_1, a_3, b_4\}$ and we observe that $\Pi
    \sm Q$ is a tree.  So $S = Q \cap V(G)$ is a strong third color of
    $G$ such that $T \subseteq S$ and $S \cap R = \emptyset$.
  \end{proofclaim}
\end{proof}

\begin{lemma}
  \label{th:tcPetersen}
  Let $G$ be the Petersen graph and $(R, T)$ be an admissible pair
  of~$G$.  Then $G$ admits a third color $S$ (possibly not strong)
  such that $T \subseteq S$ and $S \cap R = \emptyset$.  Furthermore,
  $S$ can be found in time ${\cal O}(1)$.
\end{lemma}

\begin{proof}
    We use our notation for the Petersen graph: $V(G) = \{a_1, \dots,
    a_5,$ $b_1, \dots, b_5\}$. Let $v$ be a node of $G$ and $(R, T)$
    be admissible with respect to $G, v$.  We may assume $v = a_1$
    since the Petersen graph is vertex-transitive. Since $v$ has
    degree three, we just have to study the following two cases:

    Suppose $T = N(a_1)$ and $R = \{a_1\}$.  Then we put $S = T$ and
    we observe that $G \sm S$ is a $C_6$ plus an isolated node. So $S$
    is a third color of $G$ such that $T \subseteq S$ and $S \cap R =
    \emptyset$.  Note that in this case there exists no strong third
    color that satisfies our constraints.

    Suppose $T = \emptyset$ and $R = N[v]$.  Then we put $S= \{a_3,
    b_3, b_5\}$ and we observe that $G \sm S$ is a tree.  So $S$ is a
    third color of $G$ such that $T \subseteq S$ and $S \cap R =
    \emptyset$.
 \end{proof}

\begin{lemma}
  \label{th:tcSP}
  Let $G$ be a non-basic, connected, triangle-free, square-free and
  Petersen-free graph in $\cal C$ that has no 1-cutset and no proper
  1-join.  Let $(R, T)$ be an admissible pair of~$G$.  Then $G$ admits
  a strong third color $S$ such that $T \subseteq S$ and $S \cap R =
  \emptyset$.  Furthermore, $S$ is obtained in time ${\cal O}(1)$ from
  well chosen strong third colors of blocks of $G$ w.r.t.\ a proper
  2-cutset of~$G$.
\end{lemma}

\begin{proof}
  By Theorem~\ref{th:nosquare}, $G$ has a proper 2-cutset.  Let $(X,
  Y, a, b)$ be a split of a proper 2-cutset of $G$. Let $v$ be a node
  of $G$ and $(R, T)$ an admissible pair with respect to $G, v$.  We
  now show that $G$ admits a strong third color $S$ such that $T
  \subseteq S$ and $S \cap R =\emptyset$.  We use induction on the
  blocks of decomposition $G_X$ and $G_Y$ w.r.t.\ this proper
  2-cutset, as defined in Section~\ref{sec:decomp}.  Note that by
  Lemma~\ref{l0}, $G_X$ and $G_Y$ are connected, triangle-free,
  square-free, Pertersen-free, contain no 1-cutset and no proper
  1-join.

  Here below, when we write ``by induction'', we mean that either we
  use inductively Lemma~\ref{th:tcSP} for a smaller graph (when this
  smaller graph is not basic), or that we use Lemma~\ref{th:tcbasic}
  (when this smaller graph is basic).  By symmetry it is enough to
  consider the following three cases.

  \noindent{\bf Case 1:} $v = a$. 

  Since $a$ is not of degree two, either $T = N(a)$ and $R = \{a\}$,
  or $T = \emptyset$ and $R = N[a]$.

  Suppose that $T = N(a)$. By induction there is a strong third color
  $S_X$ of $G_X$ (resp. $S_Y$ of $G_Y$) such that ${N_{G_X}}(a)
  \subseteq S_X$ (resp. ${N_{G_Y}}(a) \subseteq S_Y$). So marker node
  $c \in S_X \cap S_Y$, and hence neither $a$ nor $b$ belongs to $S_X
  \cup S_Y$.  Therefore $S=S_X \cup S_Y$ with possibly $c$ removed if
  $c$ is not a real node of $G$, is a stable set of $G$ such that $T
  \subseteq S$ and $S \cap R=\emptyset$ (since $R=\{ a\}$).  Let $H$
  be a cycle of $G$. If $H$ contains $a$, then it must contain a node
  of $N(a) = T$, and hence it contains a node of $S$. So assume that
  $H$ does not contain $a$.  Since $H$ does not contain $a$,
  w.l.o.g.\ $V(H) \subseteq X \cup \{ b \}$ and does not contain $c$.
  Hence $H$ is a cycle of $G_X$ that does not contain~$c$.  Since
  $S_X$ is a strong third color of $G_X$, a node of $H$ belongs to
  $S_X \setminus \{ c\}$, and hence to~$S$.  Therefore $S$ is a strong
  third color of $G$ such that $T \subseteq S$ and $S \cap
  R=\emptyset$.

  Now suppose that $R = N[a]$. Note that since $c$ is of degree two in
  $G_X$, $(N_{G_X} [a],\{ b \} )$ is an admissible pair
  w.r.t. $G_X,c$, and hence by induction, there exists a strong third
  color $S_X$ of $G_X$ such that $N_{G_X} [a] \cap S_X=\emptyset$ (in
  particular, $c \not\in S_X$) and $b \in S_X$. Similarly, there
  exists a strong third color $S_Y$ of $G_Y$ such that $N_{G_Y} [a]
  \cap S_Y=\emptyset$ (in particular, $c \not\in S_Y$) and $b \in
  S_Y$.  Clearly $S = S_X \cup S_Y$ is a stable set such that
  $\emptyset = T \subseteq S$ and $S \cap R =\emptyset$.  Let us check
  that every cycle of $G$ contains a node of $S$. Let $H$ be a cycle
  of $G$. If $H$ contains $b$ then we are done since $b\in S$, so
  w.l.o.g.\ $V(H) \subseteq X \cup \{ a\}$, i.e. $H$ is a cycle of
  $G_X$, and hence, since $S_X$ is a strong third color of $G_X$,
  $S_X$ contains a node of $H$, and so does $S$.

  \noindent{\bf Case 2:} $v$ is of degree two, $N(v) = \{u, w\}$, $v
  \in X$, $w = a$ and $N[w] \subseteq R$.

  Note that either $u = b$ or $u \in X$, and either $u \in R$ or $u
  \in T$.  By induction there exists a strong third color $S_X$ of
  $G_X$ such that $N_{G_X}[w] \cap S_X = \emptyset$ and $u \in S_X$ if
  and only if $u \in T$. By induction, since $c$ is of degree~2 (and
  hence both $(N_{G_Y} [w],\{ b \})$ and $(N_{G_Y} [w] \cup \{ b \},
  \emptyset )$ are admissible w.r.t. $G_Y,c$), there exists a strong
  third color $S_Y$ of $G_Y$ such that $N_{G_Y}[w] \cap S_Y
  =\emptyset$ and $b\in S_Y$ if and only if $b\in S_X$. Clearly $S =
  S_X\cup S_Y$ is a stable set of $G$ such that $T \subseteq S$ and $R
  \cap S = \emptyset$. Since $c \notin S_X$ and $c\notin S_Y$, it is
  easy to see that $S$ contains a node of every cycle of $G$, i.e. $S$
  is a strong third color of $G$.

  \noindent{\bf Case 3:} $T \cup R \subseteq X \cup \{ a,b\}$. 
    
  By induction there exists a strong third color $S_X$ of $G_X$ such
  that $T \subseteq S_X$ and $R \cap S_X = \emptyset$. If $c \in S_X$
  and $c$ is a real node of $G$, then let $T_Y=N_{G_Y}(a)$ and $R_Y=\{
  a \}$. Note that $(R_Y,T_Y)$ is an admissible pair
  w.r.t.\ $G_Y,a$. In all other cases, let $T_Y=S_X \cap \{ a,b \}$
  and $R_Y=\{ c\} \cup (\{ a,b \} \setminus S_X)$.  Note that
  $(R_Y,T_Y)$ is an admissible pair w.r.t.\ $G_Y,c$.  By induction
  there exists a strong third color $S_Y$ of $G_Y$ such that $T_Y
  \subseteq S_Y$ and $R_Y \cap S_Y = \emptyset$. Note that $S_X \cap
  \{ a,b \} = S_Y \cap \{ a,b \}$. Furthermore, if $c \in S_X$ and $c$
  is a real node of $G$, then $c \in S_Y$, and in all other cases $c
  \not\in S_Y$ because $c\in R_Y$.  If $c \in S_X$ and $c$ is a real
  node of $G$, then let $S=S_X \cup S_Y$, and otherwise let $S=(S_X
  \cup S_Y) \setminus \{ c \}$. Clearly $S$ is a stable set of $G$
  such that $T \subseteq S$ and $R \cap S =\emptyset$.

  Let $H$ be a cycle of $G$. We now show that $S$ contains a node of
  $H$.  If $H$ is a cycle of $G_X$ (resp. $G_Y$) then $S_X$
  (resp. $S_Y$) contains a node of $H$, and hence so does $S$. So we
  may assume that $H$ is not a cycle of $G_X$ nor $G_Y$. In particular
  $H$ contains both $a$ and $b$, a node of $X$ and a node of $Y$.  Let
  $H_Y$ be the $ab$-subpath of $H$ whose intermediate nodes belong to
  $Y$. Note that $H_Y \neq acb$, since otherwise $H$ belongs to
  $G_X$. So $V(H_Y) \cup \{ c \}$ induces a cycle of $G_Y$.  Since
  $S_Y$ is a strong third color of $G_Y$, it contains a node $h$ of
  $V(H_Y) \cup \{ c \}$. If $h \neq c$ then $h \in S \cap V(H)$. So
  assume that $h=c$. But then $c\notin R_Y$ so $c \in S_Y$, and hence
  $c \in S_X$ and it is a real node of $G$. Therefore, $c \in S \cap
  V(H)$.
\end{proof}

Lemma~\ref{th:tcSP} implies a weaker statement: the existence of a
third color (possibly not strong). But we do not know how to prove
this weaker result with the weaker induction hypothesis. An attempt
fails at the proof of Case~3.

\begin{lemma}
  \label{th:thirdcolor}
  Let $G$ be a non-basic connected triangle-free graph in $\cal C$ and
  $(R, T)$ be an admissible pair of~$G$.  Then $G$ admits a third
  color $S$ such that $T \subseteq S$ and $S \cap R = \emptyset$.
  Furthermore, $S$ is obtained in time ${\cal O}(1)$ from well chosen
  third colors of blocks of $G$ w.r.t.\ a 1-cutset, a proper 1-join or
  a proper 2-cutset of~$G$.
\end{lemma}

\begin{proof}
  Here below, we use the fact that every strong third color is a third
  color with no explicit mention. So, we may assume that $G$ contains
  a square or the Petersen graph, or has a 1-cutset or a proper
  1-join, for otherwise the result follows from Lemma~\ref{th:tcSP}.
  Hence, by Theorem~\ref{th:square}, the proof follows from the
  following two claims.

  Here below, when we write ``by induction'', we mean that either we
  use inductively Lemma~\ref{th:thirdcolor} for a smaller graph (when
  this smaller graph is not basic), or that we use
  Lemma~\ref{th:tcbasic} or~\ref{th:tcPetersen} (when this smaller 
  graph is basic).
  
  \begin{claim}
    \label{c:1-cutset}
    The lemma holds when $G$ has a 1-cutset.
  \end{claim}

  \begin{proofclaim}
    Let $(X,Y,z)$ be a split of a 1-cutset of $G$. Let $G_X$ and $G_Y$
    be the blocks of decomposition w.r.t.\ this 1-cutset.  Note that
    $G_X$ and $G_Y$ are triangle-free by Lemma~\ref{l:Htriangle}.

    \noindent{\bf Case~1: } $X \cap (R \cup T)$ and $Y
    \cap (R \cup T)$ are both non-empty. 

    Then, $z\in R \cup T$. We put $R_X = R \cap (X\cup \{ z \} )$,
    $R_Y = R \cap (Y\cup \{ z \} )$, $T_X = T \cap (X\cup \{ z \} )$,
    $T_Y = T \cap (Y\cup \{ z \} )$.  We observe that $(R_X, T_X)$ and
    $(R_Y, T_Y)$ are admissible with respect to $G_X$ and $G_Y$
    respectively.  So by induction there exists a third color $S_X$ of
    $G_X$ such that $S_X \subseteq T_X$, $R_X \cap S_X = \emptyset$,
    and a third color $S_Y$ of $G_Y$ such that $T_Y \subseteq S_Y$,
    $R_Y \cap S_Y = \emptyset$. So, $S=S_X \cup S_Y$ is a third color
    of $G$ such that $T \subseteq S$ and $R \cap S = \emptyset$.

    \noindent{\bf Case~2: } One of $X \cap (R \cup T)$, $Y \cap (R
    \cup T)$ is empty.

    We assume w.l.o.g.\ that $Y \cap (R \cup T) = \emptyset$. Hence, $R
    \cup T \subseteq X \cup \{z\}$.  Let $S_X$ be a third color
    of $G_X$ such that $S_X \subseteq T$, $R \cap S_X = \emptyset$. If
    $z \in S_X$, let $S_Y$ be a third color of $G_Y$ such that
    $z \in S_Y$. Else, let $S_Y$ be a third color of $G_Y$ such
    that $z \notin S_Y$.  In either case, $S=S_X \cup S_Y$ is a third
    color of $G$ such that $S \subseteq T$ and $R \cap S = \emptyset$.
  \end{proofclaim}

  \begin{claim}\label{1join:2}
    The lemma holds when $G$ has a proper 1-join.
  \end{claim}

  \begin{proofclaim}
    Let $(X, Y, A, B)$ be a split of a proper 1-join of $G$.  We show that
    $G$ admits a third color $S$ such that $T \subseteq S$ and $S \cap
    R = \emptyset$.

    Suppose that $v$ is of degree at least three. Then we assume
    w.l.o.g.\ $v\in X$.  So $v \in V(G_X)$, $T \cap Y=B$ or
    $\emptyset$, and $R \cap Y = B$ or $\emptyset$.

    Suppose that $v$ is of degree~2 and $N(v) = \{u, w\}$.  If $\{u,
    v, w\}$ is contained in $X$ or $Y$, then we assume w.l.o.g.\ that
    it is contained in $X$.  Otherwise, $v$ must be contained in $A
    \cup B$, and we assume w.l.o.g.\  that $v \in B$, which implies
    that $A=\{ u,w\}$ (since $|A| \geq 2$).  If $v \in B$, we assume
    that the marker $y$ of block $G_X$ is $v$.

    So in all cases $v \in V(G_X)$, $T \cap Y=B$ or $\emptyset$, $R
    \cap Y=B$ or $\emptyset$, and if $N(v)=\{ u,w\}$ then $v \not \in
    A$ and $u,w \in X$.  If $T\cap Y=B$ then let $T_X=(T \setminus B)
    \cup \{ y \}$, and if $T \cap Y=\emptyset$ then let $T_X=T$.  If
    $R\cap Y=B$ then let $R_X=(R \setminus B) \cup \{ y \}$, and if $R
    \cap Y=\emptyset$ then let $R_X=R$.  Note that $(R_X,T_X)$ is an
    admissible pair w.r.t.\ $G_X,v$.  By induction, there exists a
    third color $S_X$ of $G_X$ such that $T_X \subseteq S_X$ and $R_X
    \cap S_X=\emptyset$.  By induction, there exists a third color
    $S_Y'$ of $G_Y$ such that $N(x)=B \subseteq S_Y'$, and a third
    color $S_Y''$ of $G_Y$ such that $S_Y'' \cap N[x]=\emptyset$. If
    $y \in S_x$ then let $S_Y=S_Y'$, and otherwise let
    $S_Y=S_Y''$. Note that $x \not \in S_Y$, i.e. $S_Y \subseteq
    Y$. Let $S=(S_X \cap X) \cup S_Y$.  Note that only one of $S'_Y,
    S''_Y$ needs to  be computed once $S_X$ is known. 

    Clearly $S$ is a stable set.  If $T \cap Y=\emptyset$ then
    $T_X=T$, and hence, since $T_X \subseteq S_X$, $T \subseteq S$.
    If $T \cap Y=B$ then $y \in T_X$, and hence, since $T_X \subseteq
    S_X$ (and in particular $y\in S_X$), $B \subseteq S$, and
    therefore $T \subseteq S$.  If $R \cap Y=\emptyset$ then $R_X=R$,
    and hence, since $R_X \cap S_X =\emptyset$, $R \cap S=\emptyset$.
    If $R \cap Y=B$ then $R_X=(R\setminus B) \cup \{ y \}$, and hence,
    since $R_X \cap S_X =\emptyset$, $y \not \in S_X$ and so $S_Y \cap
    N[x] =\emptyset$, implying that $R \cap S=\emptyset$.

    So it only remains to show that $S$ contains a node of every odd
    cycle of $G$. Let $H$ be an odd cycle of $G$.  If $V(H) \subseteq
    X$, then since $S_X$ is a third color of $G_X$, $S_X$ contains a
    node of $H$, and hence so does $S$.  If $V(H) \subseteq Y$, then
    since $S_Y$ is a third color of $G_Y$, $S_Y$ contains a node of
    $H$, and hence so does $S$.  So we may assume that $H$ contains
    both a node of $X$ and a node of $Y$.  Hence, $H$ is node-wise
    partitioned into a path of $X$ and a path of $Y$ of different
    parity. Hence if we suppose that $H$ is minimal with respect to
    the property of overlapping $X, Y$ and being odd, then either
    $V(H) \cap X=\{ h_X\} \subseteq A$ or $V(H) \cap Y=\{ h_Y\}
    \subseteq B$.  Suppose that $V(H) \cap X=\{ h_X\}$. Then $(V(H)
    \setminus \{ h_X\} ) \cup \{ x\}$ induces an odd cycle $H'$ of
    $G_Y$. Since $S_Y$ is a third color of $G_Y$, $S_Y$ contains a
    node $h$ of $H'$. Since $x \not \in S_Y$, $h$ is a node of $V(H)
    \cap S$. Finally assume that $V(H) \cap Y=\{ h_Y\}$. Then $(V(H)
    \setminus \{ h_Y\} ) \cup \{ y\}$ induces an odd cycle $H'$ of
    $G_X$. Since $S_X$ is a third color of $G_X$, $S_X$ contains a
    node $h$ of $H'$. If $h \neq y$ then $h$ is a node of $V(H) \cap
    S$. So assume $h=y$.  Then $y \in S_X$ and hence $B \subseteq
    S_Y$, and in particular $h_Y \in V(H) \cap S$.
  \end{proofclaim}
\end{proof}

Our proof of Lemmas~\ref{th:tcSP}, \ref{th:thirdcolor} suggests that
for every triangle-free graph in $\cal C$ there might exist a stable
set that intersects every cycle. Such a property might be of use for
stronger notions of coloring (list coloring, \dots). It holds for
every basic graph (even for the Petersen graph), for every
square-and-Petersen-free graph by a slight variant of
Lemma~\ref{th:tcSP} and we almost proved it in general. But it is
false. Let us build a counter-example $G$, obtained from four disjoint
copies $\Pi_1, \dots, \Pi_4$ of the Petersen graph minus one node. So
$\Pi_i$ contains a set~$X_i$ of three nodes of degree two ($i= 1,
\dots, 4$). We add all edges between $X_1, X_2$, between $X_2, X_3$,
between $X_3, X_4$ and between $X_4, X_1$.  Note that $G$ can be
obtained by gluing one square $S=s_1s_2s_3s_4$ and four disjoint
copies $\Pi_1, \Pi_2, \Pi_3, \Pi_4$ of the Petersen graph along
Operation~${\cal O}_2$ of Theorem~\ref{th:structure} as follows: let
$G=S$ and for $i=1$ to $i=4$, replace $G$ by itself glued through
$s_i$ with $\Pi_i$.  So $G\in {\cal C}$ by Theorem~\ref{th:structure}.

We claim now that $G$ does not contain a stable set that intersects
every cycle. Indeed, if $S$ is such a stable set then $S$ must contain
all nodes in one of the $X_i$'s for otherwise we build a $C_4$ of
$G\sm S$ by choosing a node in every~$X_i$. So $X_1 \subseteq S$
say. We suppose that $\Pi_1$ has nodes $\{a_2, \dots, a_5, b_1, \dots,
b_5\}$ with our usual notation. So $X_1 = \{b_1, a_2, a_5\} \subseteq
S$ and we observe that every node in $C = V(\Pi_1\sm S)$ has a
neighbor in $X_1$. Hence $C \cap S = \emptyset$ while $G[C]$ is a
cycle on six nodes, a contradiction.  Note that for any node $v$ of
$G$, $G\sm v$ contains a stable set that intersects every cycle.  So,
a characterisation by forbidding induced subgraphs of the class of
graphs that admit a stable set intersecting every cycle needs to
consider $G$ somehow.  For this reason, we believe that such a
characterisation must be complicated.

\begin{theorem}
  If $G \in {\cal C}$, then either $\chi(G) = \omega(G)$ or
  $\chi(G) \leq 3$. In particular, $\chi(G) \leq \omega(G) + 1$.   
\end{theorem}

\begin{proof}
  Clearly we may assume that $G$ is connected.  If $\omega(G) \leq 2$
  then $\chi(G) \leq 3$ since $G$ contains a third color by
  Lemma~\ref{th:tcbasic}, \ref{th:tcPetersen} or~\ref{th:thirdcolor}
  (indeed, every non-empty graph has an admissible pair:
  $(N[v],\emptyset )$).  If $\omega(G) \geq 3$ then by
  Theorem~\ref{th:triangle}, $G$ admits a 1-cutset. So every
  2-connected component of $G$ is either a clique or is
  3-colorable. Hence $\chi(G) = \omega(G)$.
\end{proof}

\begin{theorem}
  There exists an algorithm that computes an optimal coloring of any
  graph in $\cal C$ in time ${\cal O}(nm)$. 
\end{theorem}

\begin{proof}
  Let $G$ be a graph in ${\cal C}$.  When $G$ has a 1-cutset, then it
  is easy to obtain an optimal coloring of $G$ from optimal colorings
  of its blocks. So by Theorem~\ref{th:triangle} we may assume that
  our input graph is triangle-free. We may also assume that $G$ is
  connected and not bipartite.  Now we show how to 3-color $G$ by
  finding a third color of $G$.

  We first construct a proper decomposition tree $T_G$ of the input
  graph $G$ in time ${\cal O}(nm)$, by Theorem~\ref{dtalg}.  Let $v$
  be any node of $G$.  We associate with node $G$ of $T_G$ an
  admissible pair $(R,T)$, say $R=N[v]$ and $T=\emptyset$, and we use
  Lemmas \ref{th:tcbasic}, \ref{th:tcPetersen}, \ref{th:tcSP} and
  \ref{th:thirdcolor} to recursively find a third color $S$ of $G$
  such that $T \subseteq S$ and $S \cap R =\emptyset$.

  First note that all the leaves of $T_G$ are basic, and since $G$ is
  triangle-free, they are one of the graphs described in
  Lemma~\ref{th:tcbasic} or \ref{th:tcPetersen}.  So once the
  appropriate admissible pairs have been associated with a given leaf
  of $T_G$, Lemma~\ref{th:tcbasic} or \ref{th:tcPetersen} shows how to
  find the appropriate third color (or if needed, strong third
  colors), each in linear time in the size of the leaf.  

  For a non-leaf node $H$ of type~1 of $T_G$,
  Lemma~\ref{th:thirdcolor} shows how to proceed (in linear time) to
  find a third color of $H$ by asking recursively for appropriately
  chosen third colors of its children (i.e. choosing appropriate
  admissible pairs to associate with its children, and once the
  appropriate third colors of its children are found how to put them
  together to find the desired third color of $H$).  Note that in
  several cases, the algorithm has to compute the third colors of the
  children $H_1, H_2$ of $H$ in a prescribed order, that is wait for
  the answer for the coloring of $H_1$ before knowing what admissible
  pair is needed for the coloring of $H_2$.

  For a non-leaf node $H$ of type~2 of $T_G$, to recursively find a
  third color of $H$, we actually need to find a strong third color.
  This we can do by proceeding as in the proofs of Lemma
  \ref{th:tcSP}.  Note that since $H$ is of type~2, every non-leaf
  descendant of $H$ is of type~2. Also, no descendant of $H$ contains
  the Petersen graph, so all leaves under $H$ will have a strong third
  color computed by Lemma~\ref{th:tcbasic}.

  So the processing time at each non-leaf node of $T_G$ is ${\cal O}
  (1)$. Since by Lemma~\ref{l:sizeT} the size of $T_G$ is ${\cal O}
  (n)$, the sum of processing times at the leaves of $T_G$ is ${\cal
    O} (n+m)$.  So the time needed to process the tree is ${\cal O}(n+m)$.

  Hence, the total computation time is ${\cal O} (nm)$.
\end{proof}

We note that the algorithm above has complexity ${\cal O}(n+m)$ once
the proper decomposition tree is given.  So, if one can find an ${\cal
  O}(n+m)$-time algorithm for constructing a proper decomposition
tree, then one gets an ${\cal O}(n+m)$-time coloring algorithm for
graphs in $\cal C$.

\section{Cliques and stable sets}
\label{s:cliquestable}

The problems of finding a maximum clique and a maximum stable set for
a graph in $\cal C$ have a complexity that is easy to establish.  Note
first that any graph in $\cal C$ contains no diamond.  So any edge of
a graph in $\cal C$ is contained in a unique maximal clique.  Hence,
to find a maximum clique in polynomial time, it is enough to look for
the common neighborhood of every edge.  A simple implementation of
this leads to an ${\cal O}(nm)$-time algorithm and no faster algorithm
seems to be known for diamond-free graphs.  Here below we give a
linear time algorithm restricted to our class ${\cal C}$.  Note in
contrast that, as proved by Maffray and Preissmann, the coloring
problem is NP-hard even when restricted to triangle-free
graphs~\cite{maffray.preiss:triangle}.

\begin{theorem}
  There exists a linear time algorithm whose input is a graph in $\cal
  C$ and whose output is a maximum clique of $G$.
\end{theorem}

\begin{proof}
  It is trivial to decide in linear time whether $\omega(G) = 1$. So,
  we assume $\omega(G) \geq 2$.  Then the 2-connected components of
  $G$ can be found in linear time (see~\cite{hopcroft.tarjan:447}
  and~\cite{tarjan:dfs}). By Theorem~\ref{th:triangle}, every
  2-connected component is either a clique or is triangle-free. So, to
  find a maximum clique it suffices for each 2-connected component to
  test whether it is a clique or not, and to output a largest such
  clique (if any). If no 2-connected component is a clique then output
  any edge.
\end{proof}

A \emph{2-subdivision} is a graph obtained from any graph by
subdividing twice every edge. More precisely, every edge $uv$ of a
graph $G$ is replaced by an induced path $uabv$ where $a$ and $b$ are
of degree two. Let $F$ be the resulting graph. It is easy to see that
$\alpha(F) = \alpha(G) + |E(G)|$. This construction, due to Poljak,
easily yields:

\begin{theorem}[Poljak \cite{poljak:74}]
  The problem whose instance is a 2-subdivision $G$ and an integer $k$
  and whose question is ``Does $G$ contain a stable set of size at
  least $k$'' is NP-complete. 
\end{theorem}

Since every 2-subdivision is in $\cal C$, a direct consequence is:

\begin{theorem}
  Finding a maximum stable set of a graph in $\cal C$ is NP-hard.
\end{theorem}

\section{NP-completeness of $\Pi_{H_{3|3}}$}
\label{s:bounds}

It is mentioned in Section~\ref{f:sgraph} that the problem
$\Pi_{H_{3|3}}$ is NP-complete. To prove this, we need a slight
variation on a theorem proved in \cite{leveque.lmt:detect} using a
refinement of a construction due to
Bienstock~\cite{bienstock:evenpair}.  We remind the reader that $I$
denotes the graph on nodes $a, b, c, d, e, f$ with the following
edges: $ab$, $ac$, $ad$, $be$, $bf$.

For any integer $k\geq 1$ we define the problem $\Pi_k$ whose instance
is described below:

\begin{enumerate}
\item
  a graph $G$ that does not contain $I$;
\item
  two nodes of degree two $x$ and $y$;  
\item
  \label{i:lengthk}
  and these are such that there exists a path $P_x$ (resp. $P_y$),
  on $2k-1$ nodes, all interior nodes of which are of degree 2,
  and $x$ (resp. $y$) is the middle of $P_x$ (resp. $P_y$).  Moreover,
  $P_x$ and $P_y$ are node-disjoint.  
\end{enumerate}

The question of $\Pi_k$ is ``does there exist an induced cycle of $G$
that goes through $x, y$?''.

\begin{theorem}[L\'ev\^eque, Lin, Maffray and Trotignon \cite{leveque.lmt:detect}]
  \mbox{} \label{th:linetal} 
  For all integers $k\geq 1$ the problem $\Pi_k$ is NP-complete.
\end{theorem}

\begin{proof}
  The theorem is not stated this way in~\cite{leveque.lmt:detect}, so
  we explain here how to obtain a proof
  from~\cite{leveque.lmt:detect}.  In this article, Theorem~2.3 states
  that Problem~$\Pi$, that is Problem $\Pi_k$ where
  Condition~\ref{i:lengthk} is forgoten, is NP-complete.  The proof is
  obtained by a reduction from 3-SAT.  For every integer $k$ and every
  instance $f$ of 3-SAT, an instance denoted by $G_{f}(k, 1, 1, 1, 3,
  1, 1)$ of $\Pi$ is built.  It is shown that $f$ has a truth
  assignment if and only if $G_{f}(k, 1, 1, 1, 3, 1, 1)$ contains an
  induced cycle that goes through $x, y$.  It turns out that luckily
  $G_{f}(k, 1, 1, 1, 3, 1, 1)$ is an instance of $\Pi_k$ (even if
  unluckily this is not explicitly stated
  in~\cite{leveque.lmt:detect}), because $G_{f}(k, 1, 1, 1, 3, 1, 1)$
  satisfies Condition~\ref{i:lengthk}.
\end{proof}

\begin{theorem}
  \label{th:NPC}
  Problem $\Pi_{H_{3|3}}$ is NP-complete. 
\end{theorem}

\begin{proof}
  We use the fact that every realisation of ${H_{3|3}}$ contains $I$
  (note that this is false for ${H_{1|1}}$, ${H_{2|1}}$, ${H_{3|1}}$,
  ${H_{2|2}}$ and ${H_{3|2}}$).  Let $G, x, y$ be an instance of the
  NP-complete problem $\Pi_3$.  So, $G$ is a graph that does not
  contain $I$ and $x$, $y$ are two nodes of $G$ of degree two.

  We prepare now an instance $G'$ of $\Pi_{H_{3|3}}$ by adding an edge
  between $x, y$.  We note that $G'$ contains a unique $I$, induced by
  $N(x) \cup N(y)$.  So, any realisation of $H_{3|3}$ in $G'$ must
  contain $P_x$ and $P_y$.  It follows that $G$ contains a hole
  passing through $x, y$ if and only if $G'$ contains a realisation of
  ${H_{3|3}}$.
\end{proof}

\section*{Acknowledgement}

We are grateful to Vincent Limouzy for pointing to us that 1-joins can
be found in linear time.

\end{document}